\documentclass[pdflatex,sn-mathphys]{sn-jnl}

\usepackage{amsmath}
\usepackage{amssymb}
\usepackage{amsfonts}
\usepackage{amsthm}
\usepackage{mathtools}
\usepackage{enumerate}
\usepackage{marginnote}
\usepackage{tikz}
\usetikzlibrary{patterns}
\usepackage{lmodern}

\usepackage{stmaryrd}
\SetSymbolFont{stmry}{bold}{U}{stmry}{m}{n}

\usepackage{graphicx}      
\usepackage{natbib}        

\usepackage{proofread}    

\usepackage{amsmath,bbm,bm,mathtools}    
\usepackage{enumitem}
\usepackage{graphicx}   
\usepackage{verbatim}   
\usepackage{latexsym}
\usepackage{amsfonts}
\usepackage{amssymb}
\usepackage{stmaryrd}
\usepackage[utf8]{inputenc}
\usepackage{array}
\usepackage{pgf}
\usepackage{tikz}
\usetikzlibrary{arrows,automata}
\usepackage{braket} 
\usepackage{mathrsfs}
\usepackage{enumitem}
\usepackage{wrapfig}

\usepackage[colorinlistoftodos]{todonotes}

\newcommand{\fA}{\mathfrak{A}}

\newcommand{\ddt}{\frac{\text{\normalfont d}}{\text{\normalfont d}t}}
\newcommand{\ddts}{\tfrac{\text{\normalfont d}}{\text{\normalfont d}t}}

\newcounter{i} 
\newtoks\striche 

\newcommand{\Sc}{\mathcal{S}}

\newcommand{\Rf}{\mathbf{R}}

\newcommand{\N}{{\mathbb{N}}}
\newcommand{\R}{{\mathbb{R}}}
\newcommand{\C}{{\mathbb{C}}}

\newcommand{\X}{\mathcal{X}}
\newcommand{\Y}{\mathcal{Y}}
\newcommand{\U}{\mathcal{U}}
\newcommand{\V}{\mathcal{V}}
\newcommand{\W}{\mathcal{W}}
\newcommand{\Z}{\mathcal{Z}}

\let\Re\relax

\DeclareMathOperator{\Re}{Re}

\DeclareMathOperator{\rk}{rank}

\DeclareMathOperator{\dom}{dom}

\DeclareMathOperator{\im}{im}

\newcommand{\id}{I}
\newcommand{\Lb}{\mathcal{L}_{\mathrm{b}}} 
\newcommand{\Lp}[1]{\mathrm{L}^{#1}} 
\newcommand{\Wkp}[1]{\mathrm{W}^{#1}} 

\newcommand{\conC}{\mathrm{C}} 

\newcommand{\adjunsymb}{\ast} 

\newcommand{\adjun}[1][1]{%
  \setcounter{i}{1}%
  \striche={\adjunsymb}%
  \loop%
  \ifnum\value{i}<#1%
  \striche=\expandafter{\the\expandafter\striche\adjunsymb}%
  \stepcounter{i}%
  \repeat%
  ^{\the\striche}%
}

\newcommand{\bvek}[2]
{
   \begin{bmatrix}
      #1\\
      #2
   \end{bmatrix}
}

\newcommand{\sbvek}[2]{{\left[\begin{smallmatrix}#1\\#2\end{smallmatrix}\right]}}

\newcommand{\spvek}[2]{\left(\begin{smallmatrix}#1\\#2\end{smallmatrix}\right)}

\newcommand{\sbmat}[4]{\left[\begin{smallmatrix}#1 & #2\\#3 & #4\end{smallmatrix}\right]}

\DeclarePairedDelimiter{\norm}{\lVert}{\rVert}

\DeclarePairedDelimiterX{\setdef}[2]{\{}{\}}{#1\,\delimsize\vert\,\mathopen{} #2}
\DeclarePairedDelimiterX{\scprod}[2]{\langle}{\rangle}{#1,#2}
\DeclarePairedDelimiterX{\dualprod}[2]{\langle}{\rangle}{#1,#2}
\DeclarePairedDelimiterX{\sdprod}[2]{\llangle}{\rrangle}{#1,#2} 
\newcommand{\dd}{\mathrm{d}} \newcommand{\dx}[1][x]{\mathop{\dd#1}}

\jyear{2025}%

\theoremstyle{thmstyleone}%
\newtheorem{theorem}{Theorem}[section]%
\newtheorem{proposition}[theorem]{Proposition}%
\newtheorem{lemma}[theorem]{Lemma}%

\theoremstyle{thmstyletwo}%
\newtheorem{remark}[theorem]{Remark}%

\newcommand{\smalldiamond}{\raisebox{0.15ex}{\scalebox{0.8}{$\lozenge$}}}

\makeatletter
\newcommand{\inlineqed}[1]{%
  \leavevmode\unskip\penalty9999\hskip0pt\hbox{}\nobreak\hfill
  \hbox{#1}%
}
\let\oldremark\remark
\let\endoldremark\endremark
\renewenvironment{remark}{%
  \oldremark%
}{%
  \inlineqed{\smalldiamond}\endoldremark%
}
\makeatother
\theoremstyle{thmstylethree}%
\newtheorem{definition}[theorem]{Definition}%

\begin{document}

\title[Infinite-dimensional dissipation inequality]{The infinite-dimensional dissipation inequality}

\author{\fnm{Timo} \sur{Reis}}
\email{timo.reis@tu-ilmenau.de}

\affil{\orgdiv{Institut f\"ur Mathematik}, \orgname{Technische Universit\"at Ilmenau}, \orgaddress{\street{Weimarer Str.\ 25}, \city{Ilmenau}, \postcode{98693 }, \country{Germany}}}

\abstract{We study the dissipativity of linear infinite-dimensional systems with respect to a prescribed quadratic supply rate functional. We characterize this property via an operator inequality that also yields the system’s dissipation rate. We also derive implications for the linear–quadratic optimal control problem on the nonnegative real half-line.
}



\maketitle
\pagestyle{plain}

\section{Introduction}

A finite-dimensional linear system
\begin{equation}
\dot{x}(t)=Ax(t)+Bu(t),\quad
y(t)=Cx(t)+Du(t),
\label{eq:sys_fin}
\end{equation}
$A\in\C^{n\times n}$, $B\in\C^{n\times m}$, $C\in\C^{p\times n}$, $D\in\C^{p\times m}$, is said to be {\em dissipative with respect to the supply rate functional}
\begin{equation}
s:\C^{m+p}\times \R, \quad \spvek{y}{u}\mapsto \spvek{y}{u}\adjun\sbmat{Q}{S}{S\adjun}{R}\spvek{y}{u},\label{eq:supply_fin}
\end{equation}
$Q=Q\adjun\in\C^{p\times p}$, $S\in\C^{p\times m}$, $R=R\adjun\in\C^{m\times m}$, {\em and storage function} $\Sc:\C^n\to\R$, such that the {\em dissipation inequality}
\begin{equation}\Sc(x(T))-\Sc(x(0))+ \int_0^Ts\bigl(y(t),u(t)\bigr)\dx[t]\geq0\label{eq:dissineq_fin}\end{equation}
is fulfilled for $T>0$ and all solutions of \eqref{eq:sys_fin} on $[0,T]$. As \eqref{eq:dissineq_fin} can be interpreted as an energy balance, $\Sc$ is commonly referred to as the {\em storage function}.
This concept traces back to {\sc Willems'} seminal works \cite{Will71a,Will72a,Will72b}, where dissipativity has proven to be a crucial concept with applications spanning from stability analysis of feedback systems and linear-quadratic optimal control to physical considerations such as energy transfer and balances.

If $\Sc$ is quadratic, there exists some Hermitian $P\in\C^{n\times n}$, such that $\Sc(x)=x\adjun Px$ for all $x\in\C^n$. Now, by dividing \eqref{eq:dissineq_fin} with $T$, taking the limit $T\to0$, and using \eqref{eq:sys_fin}, we obtain that
\begin{equation}
\begin{bmatrix}A\adjun P+PA+C\adjun QC&PB+C\adjun QD+C\adjun S\\B\adjun P+D\adjun QC+S\adjun C&D\adjun QD+D\adjun S+S\adjun D+R\end{bmatrix}\geq0,\label{eq:kyp}
\end{equation}
where the inequality sign refers positive semi-definiteness of the matrix on the left hand side. On the other hand, by using basic calculus, we see that \eqref{eq:kyp} implies that the solutions of \eqref{eq:sys_fin} fulfill
the dissipation inequality \eqref{eq:dissineq_fin} with $\Sc(x)=x\adjun Px$.

 For $q$ being the rank of the matrix on the left hand side of \eqref{eq:kyp}, a~rank-revealing factorization of the left hand side of \eqref{eq:kyp} gives rise to the existence of $K\in \C^{q\times n}$, $L\in \C^{q\times m}$ with
\begin{equation}
\begin{bmatrix}A\adjun P+PA+C\adjun QC&PB+C\adjun QD+C\adjun S\\B\adjun P+D\adjun QC+S\adjun C&D\adjun QD+D\adjun S+S\adjun D+R\end{bmatrix}=\bvek{K\adjun}{L\adjun}\begin{bmatrix}K&L\end{bmatrix},\label{eq:lure}
\end{equation}
which yields that the solutions of \eqref{eq:sys_fin} further fulfill
\begin{equation*}\Sc(x(T))-\Sc(x(0))+ \int_0^Ts\bigl(y(t),u(t)\bigr)\dx[t]
=\int_0^T\norm{Kx(t)+Lu(t)}^2\dx[t],
\end{equation*}
whence $\norm{Kx(t)+Lu(t)}^2$ is called {\em dissipation rate at $t$}.
Matrix equations of the form \eqref{eq:lure}, known as \emph{Lur'e equations}, have been studied from linear-algebraic and numerical perspectives in \cite{Reis2011Lure,PoloniReis2012Deflation}.

In this article, we extend these considerations to a~broad class of infinite-dimensional systems. Here, we employ the system node approach developed {\sc Staffans} in \cite{Staf05} and other previous publications by the same author. Although it may appear daunting at first glance, system nodes are, on one hand, closely related to the representation \eqref{eq:sys_fin} of finite-dimensional linear systems. On the other hand, they enable the incorporation of partial differential equations with boundary control. Formulating concrete examples of partial differential equations with input and output is often straightforward in the system node representation, which is represented by
\begin{equation}\label{eq:ODEnode}
    \spvek{\dot{x}(t)}{y(t)}= \sbvek{A\&B\\[-1mm]}{C\&D}\,\spvek{x(t)}{u(t)},
\end{equation}
where, for some Hilbert spaces $\X$, $\U$, and $\Y$,
\begin{equation}
\begin{aligned}
A\&B:\;\;&\X\times \U\supset\dom(A\&B) \to \X,\\
C\&D:\;\;&\X\times \U\supset\dom(C\&D)\to{\Y}
\end{aligned}\label{eq:sysnodeop}
\end{equation}
 are linear operators with properties explained in the subsequent section. In contrast to the finite-dimensional case, the operators $A\&B$ and $C\&D$ do not split into separate components corresponding to the state and input, driven by the necessity of incorporating boundary control for partial differential equations. For self-adjoint and bounded $Q:\Y\to \Y$, $R:\U\to \U$, and some bounded $S:\U\to \Y$, we consider supply rate functionals of the form
\begin{equation}
s:\Y\times \U\to \R, \quad \spvek{y}{u}\mapsto \scprod*{\spvek{y}{u}}{\sbmat{Q}{S}{S\adjun}{R}\spvek{y}{u}}_{\Y\times\U}.\label{eq:supply_inf}
\end{equation}
In our analysis, we permit potentially unbounded quadratic storage functions. Specifically, we examine those represented by closed, densely defined, and semi-bounded quadratic forms $\Sc$ (definitions for these terms will be provided in this article for the sake of self-contained understanding). Dissipativity now means that the solutions of \eqref{eq:ODEnode} (solution concepts for system nodes will be introduced later on) satisfy \eqref{eq:dissineq_fin}. By using
Kato's First Representation Theorem \cite[Chap.~VI, Thm.\ 2.1]{Kato80}, there exists some self-adjoint (possibly unbounded) operator $P:\X\supset\dom(P)\to\X$, such that, loosely speaking, $\Sc(x)=\scprod{x}{Px}_\X$ for all $x\in\dom(P)$. We show that the operator inequality that generalizes \eqref{eq:kyp} now reads
\begin{multline}\label{eq:kyp_inf}
\forall \, \spvek{x}u\in\W:\quad 2\Re\scprod*{\widetilde{P}x}{A\&B \spvek{x}u}_{\V,\V\adjun}\\+\scprod*{\spvek{C\&D\spvek{x}u}u}{\sbmat{Q}{S}{S\adjun}{R}\spvek{C\&D\spvek{x}u}u}_{\Y\times\U}\geq 0,
\end{multline}
where $\V\subset \X$ and $\W\subset \dom(A\&B)$ are dense subspaces (to be specified in due course), and $\widetilde{P}:\V\to\V\adjun$ is a~certain extension of $P$. Given  that we will allow for unbounded storage functions, it is crucial to carefully examine the involved spaces and have a clear understanding of the inner and duality products being employed.

In the finite-dimensional setting it is well known that, beyond physical energy balance laws, dissipation inequalities play a central role in optimal control: one seeks to minimize
\[
\int_0^\infty s\bigl(y(t),u(t)\bigr)\dx[t]
\]
(where $s$ is as in \eqref{eq:supply_fin}) subject to \eqref{eq:sys_fin} with $x(0)=x_0\in\R^n$, possibly together with a terminal constraint. In this case \cite{Will71a} shows that the \emph{value function}, the map assigning to the initial state $x_0$ the optimal cost, is a storage function. We prove that the same relation holds in infinite dimensions. Unlike \cite{Will71a} we impose no controllability assumptions. Instead, in our optimal control setting we restrict attention to nonnegative supply rate functionals, which allows us to dispense with the terminal constraints used in \cite{Will71a}. We also assume the existence of a stabilizing state feedback in an appropriate sense.

This work is organized as follows: After introducing the notation and presenting some fundamentals on quadratic forms, we present the required basics of system nodes in Section~\ref{sec:sysnodes}. The main part of this article is thereafter presented in Section~\ref{sec:main}, where we introduce the dissipation inequality for infinite-dimensional systems. We show the equivalence between the dissipation inequality and an operator inequality of type \eqref{eq:kyp_inf}. 
We next present, in Section~\ref{sec:bndstor}, further properties that hold when the storage function is bounded. Thereafter, in Section~\ref{sec:lqr} we investigate connections to linear-quadratic optimal control problems. In particular, we show that, under the nonnegativity assumption on \eqref{eq:supply_inf}, the value function of the optimal control problem minimizing $\int_0^\infty s\bigl(y(t),u(t)\bigr),\dx[t]$ is a storage function. Finally, in Section~\ref{sec:ex} we illustrate our theory by means of two examples of partial differential equations.

To place our approach in context, we first provide a brief survey of the relevant literature. As already noted, the finite-dimensional theory goes back to {\sc Willems} \cite{Will71a,Will72a,Will72b}. Dissipation inequalities for the passive infinite-dimensional case (i.e., in the notation of \eqref{eq:supply_inf}, with $\U=\Y$, $S$ the identity, and $Q=R=0$) have been studied in \cite{Pand99,Staf02,ArSt07}.
 In particular, \cite{Staf02,ArSt07}.
 make it clear that allowing for unbounded storage functions is natural. We also note that linear–quadratic optimal control problems in the special case 
(in our notation, $Q=I$ and $R=I$) have been studied in 
\cite{OpmeerStaffans2014,OpmeerStaffans2019,Opmeer2014MTNS}, together with some remarks on more general cost functionals in 
\cite{Opmeer2014MTNS}. The analysis therein shows that, essentially, the problem 
reduces to finding a bounded, nonnegative solution $P$ of \eqref{eq:kyp_inf}. 
In \cite{Opmeer2014MTNS}, further comparisons are also made with alternative 
approaches to operator equations and inequalities in connection with optimal control, 
such as the method in \cite{Grabowski2017}, which is based on the so-called 
\emph{reciprocal system} and requires the inversion of the main operator $A$.

\subsection*{Notation and preliminaries}

We denote by $\norm{\cdot}_{\X}$ the norm on the complex Hilbert space $\X$, and by $\id_\X$. 
The symbol $\X\adjun$ stands for the {\em anti-dual} of $\X$, i.e., the space of all bounded, additive, and conjugate homogeneous functionals on $\X$. Therefore, the canonical duality product $\langle\cdot,\cdot\rangle_{\X\adjun,\X}$ (as well as the inner product $\langle\cdot,\cdot\rangle_{\X}$ in $\X$) is a symmetric sesquilinear form in the sense of \cite[Chap.~VI]{Kato80}, that is, it is linear in the first argument, and swapping the arguments results in conjugation. 
Unless stated otherwise, a Hilbert space is canonically identified with its anti-dual.

\subsubsection*{Operators and forms}
We denote by $\Lb(\X,\Y)$ the space of all bounded linear operators from the Hilbert space $\X$ to the Hilbert space $\Y$. As usual, we abbreviate $\Lb(\X)\coloneqq \Lb(\X,\X)$. The domain $\dom (A)$ of a possibly unbounded linear operator $A\colon \X\supset\dom (A) \to \Y$ is equipped with the graph norm \[\norm{x}_{\dom (A)}\coloneqq \big(\norm{x}_{\X}^2 + \norm{Ax}_{\Y}^2\big)^{1/2}.\]
A~subspace of $\dom(A)$ is called a {\em core of $A$}, if it is dense in $\dom(A)$ with respect to the norm $\norm{\cdot}_{\dom(A)}$.

The adjoint of a~densely defined linear operator $A\colon \X\supset\dom (A)\to \Y$ is $A\adjun \colon  \Y\supset\dom (A\adjun)\to \X$ with
\[
 \dom (A\adjun) = \setdef{y\in \Y}
 {\exists\, z\in \X\text{ s.t.\ }\forall\,x\in\dom (A):\scprod{y}{Ax}_{\Y} = \scprod{z}{x}_{\X}}.
\]
The vector $z\in \X$ in the above set is uniquely determined by $y\in\dom (A\adjun)$, and we set $A\adjun y=z$. 

Essential for the definition of storage functions are {\em real quadratic forms}. That is, for a~subspace $\dom(\Sc)$, a~mapping $\Sc:\X\supset\dom(\Sc)\to\R$, such that $h:\dom(h)\times \dom(h)\to\C$ with
\begin{equation}
    h(x,z)=\tfrac14\big(\Sc(x+z)-\Sc(x-z)+\imath \Sc(x+\imath z)-\imath \Sc(x-\imath z)\big)\label{eq:hsesq}
\end{equation}
is a~sesquilinear form. A~real quadratic form is called {\em densely defined}, if $\dom(\Sc)$ is dense in $\X$, and
{\em bounded}, if
\[\exists\,c>0:\quad |\Sc(x)|\leq \norm{x}_\X^2\;\forall x\in\X,\]
and {\em semi-bounded}
exists some $\lambda\in\R_{\ge0}$ such that one of \begin{equation}
    x\mapsto(\pm\Sc(x)+\lambda\norm{x}_\X^2)^{1/2}\label{eq:formnorm}
\end{equation}defines a norm on $\dom(\Sc)$. If, in addition, $\dom(\Sc)$ is complete with respect to the aforementioned norm, then $\Sc$ is called {\em closed}. A~{\em core of $\Sc$} is a~subspace of $\dom(\Sc)$, which is dense in $\dom(h)$ with respect to the aforementioned norm in $\dom(\Sc)$.
For a~closed, densely defined and semi-bounded real quadratic form $\Sc:\X\supset\dom(\Sc)\to\R$, Kato's First Representation Theorem \cite[Chap.~VI, Thm.\ 2.1]{Kato80} yields that there exists some self-adjoint $P:\X\supset\dom(P)\to\X$, such that $\dom(P)$ is a~core of $\dom(\Sc)$, and
\begin{align*}
h(x,z)&=\scprod{Px}{z}_\X\quad\forall\ x\in\dom(h),\,z\in \dom(P),\\
\dom(P)&=\setdef{z\in \X}{\exists\, w\in \X\text{ s.t.\ }h(x,z)=\scprod{x}{w}_\X}.\end{align*}
The operator $P$ is uniquely determined by $h$, and we call $P$ {\em the operator induced by $h$}. 
The fact that $\dom(P)$ is a~core of $\dom(\Sc)$ implies that $P$ extends uniquely to an operator $\widetilde{P}\in\Lb(\dom(\Sc),\dom(\Sc)\adjun)$ via
\begin{equation}
    h(x,z)=\scprod{x}{\widetilde{P}z}_{\dom(\Sc),\dom(\Sc)\adjun}\quad\forall\ x,z\in\dom(\Sc).\label{eq:Pdual}
\end{equation}

\subsubsection*{Function spaces}
We adopt the notation from the widely-used textbook \cite{AdamFour03} for Lebesgue and Sobolev spaces. In the case of function spaces with values in a Hilbert space $\X$, we include the additional marker ";$\X$" after specifying the domain. For example, the Lebesgue space of $p$-integrable $\X$-valued functions on the domain $\Omega\subset\R^d$ is denoted as $\Lp{p}(\Omega;\X)$. In this article, the integration of $\X$-valued functions is consistently interpreted in the Bochner sense, as detailed in \cite{Dies77}.
Since functions and elements of the target space $\X$ might be confused, we use “$(\cdot)$” to mark functions (e.g., $x(\cdot)\in \Lp{p}(\Omega;\X)$), whereas elements of $\X$ are e.g.\ denoted by $x$.


\section{System nodes}\label{sec:sysnodes}

We present some groundwork for systems of the form \eqref{eq:ODEnode} with operators as in \eqref{eq:sysnodeop}.
The autonomous dynamics (i.e., those with a trivial input $u=0$) of \eqref{eq:ODEnode} are determined by the so-called \emph{main operator}
$A\colon \X\supset\dom(A)\to \X$
with $\dom(A) \coloneqq \setdef{x\in\X}{\spvek x0\in\dom(A\&B)}$ and $Ax \coloneqq A\&B\spvek x0$ for all $x\in\dom(A)$.

The concept of system nodes imposes inherent requirements on the operators $A\&B$ and $C\&D$. These are designed to guarantee favorable properties and an appropriate solution concept for the dynamics described by \eqref{eq:ODEnode}.
\begin{definition}[System node]\label{def:sysnode}
  A {\em system node} on the triple $(\X,\U,\Y)$ of Hilbert spaces is a~linear operator $\sbvek{A\&B}{C\&D}$ with $A\&B:\X\times\U\supset\dom(A\&B)\to \X$, $C\&D:\X\times\U\supset\dom(C\&D)\to \Y$ satisfying the following conditions:
  \begin{enumerate}[label=(\alph{*})]
    \item $A\&B$ is closed.
    \item $C\&D\in \Lb(\dom (A\&B),\Y)$.
    \item\label{def:sysnode3} For all $u\in \U$, there exists some $x\in \X$ with $\spvek{x}{u}\in \dom(A\&B)$.
    \item The main operator $A$ is the generator of a~strongly continuous semigroup $\fA(\cdot)\colon
      \R_{\ge 0}\to \Lb(\X)$ on $\X$.
  \end{enumerate}
\end{definition}
We now expound upon the solution concepts for \eqref{eq:ODEnode}.
\begin{definition}[Classical/generalized trajectories]\label{def:traj}
Let $\sbvek{A\& B}{C\& D}$ be a~system node  on $(\X,\U,\Y)$, and let $T\in\R_{>0}$.\\
A {\em classical trajectory} of \eqref{eq:ODEnode} on $[0,T]$ is a triple
\[
(x(\cdot),u(\cdot),y(\cdot))\,\in\,\conC^{1}([0,T];\X)\times \conC([0,T];\U)\times \conC([0,T];\Y),
\]
which for all $t\in[0,T]$ satisfies \eqref{eq:ODEnode}.\\
A {\em generalized trajectory} of \eqref{eq:ODEnode} on $[0,T]$ is a~limit
 of classical trajectories of \eqref{eq:ODEnode} on $[0,T]$ in the topology of $\conC([0,T];\X)\times \Lp{2}([0,T];\U)\times \Lp{2}([0,T];\Y)$.\\
Further, we call 
\[
(x(\cdot),u(\cdot),y(\cdot))\,\in\,\conC^{1}(\R_{\ge0};\X)\times \conC(\R_{\ge0};\U)\times \conC(\R_{\ge0};\Y)
\]
a~classical/generalized trajectory of \eqref{eq:ODEnode} on $\R_{\ge0}$, if, for all $T\in\R_{>0}$, the restriction of $(x(\cdot),u(\cdot),y(\cdot))$ to $[0,T]$ is a~classical/generalized trajectory of \eqref{eq:ODEnode} on $[0,T]$.
\end{definition}

A~first result on the existence of solutions is now presented:
\begin{proposition}[{\cite[Thm.~4.3.9]{Staf05}}]\label{prop:solex}
Let $\sbvek{A\& B}{C\& D}$ be a~system node  on $(\X,\U,\Y)$, $T\in\R_{>0}$, and
\[x_0\in \X,\quad u(\cdot)\in \Wkp{2,1}([0,t];\U)\;\text{ s.t.\ }\spvek{x_0}{u(0)}\in \dom (A\&B).\]
Then there exist
$x(\cdot)\,\in\,\conC^{1}([0,t];\X)$ with $x(0)=x_0$, and $y(\cdot)\in \conC([0,t];\Y)$, such that $(x(\cdot),u(\cdot),y(\cdot))$ is a~classical trajectory
of \eqref{eq:ODEnode} on $[0,T]$.
\end{proposition}

A~special additional property is {\em well-posedness}. That is, for some (and hence any) $T>0$, there exists some $c_T>0$, such that the classical (and thus also the generalized) trajectories of \eqref{eq:ODEnode} on $[0,T]$ fulfill
  \begin{equation}
\label{eq:wp}      
    \norm{x(T)}_{\X} + \norm{y(\cdot)}_{\Lp{2}([0,T];\Y)} \leq
    c_T\big(\norm{x(0)}_{\X}
    + \norm{u(\cdot)}_{\Lp{2}([0,T];\U)}\big).
  \end{equation}
If the constant $c_T$ can be chosen independently of $T$, then we call \eqref{eq:ODEnode} \emph{infinite-time well-posed}. 
In finite dimensions, well-posedness holds automatically, while infinite-time well-posedness typically is tied to stability properties of $A$.

In the infinite-dimensional setting, well-posedness often provides a certain `comfort factor' when developing the theory, although verifying it in concrete models is frequently challenging. On the other hand, it is often unnecessary, and hence not entirely {natural}, for many arguments, especially those based on dissipativity. If one dispenses with well-posedness, however, developing the theory requires additional care.

The operator $A\&B$ can be separated into components associated with the state and the input, in accordance with the conventional framework adopted in numerous studies on infinite-dimensional systems (see, for example, \cite{TucsnakWeiss2009}). To make this separation precise, one introduces state extrapolation spaces. We do not reproduce the general theory here; a full treatment is given in \cite[Sec.~4.7]{Staf05}. However, we summarize some basic properties in the appendix.

For later use, we introduce a~special space associated to a~system node $\sbvek{A\& B}{C\& D}$ on $(\X,\U,\Y)$. Namely, we define
\begin{equation}\label{eq:Vdef}
\V:=\setdef{x\in \X}{\exists\ u\in \U\text{ s.t.\ }\spvek{x}u\in\dom(A\&B)},
\end{equation}
which is a~Hilbert space equipped with the norm
\[\norm{x}_{\V}:=\inf\setdef{\norm{\spvek{x}u}_{\dom(A\&B)}}{\exists\ u\in \U\text{ s.t.\ }\spvek{x}u\in\dom(A\&B)},\]
see \cite[Lem.~4.3.12]{Staf05}. Strictly speaking, this space should carry an index indicating its dependence on $A\&B$. In the present paper there is no risk of confusion, since the symbol $\V$ is used exclusively for this space; we therefore keep the simplified notation. The same convention will be used for the space
\begin{equation}\label{eq:Wdef}
\W:=\setdef{\spvek{x}u\in\dom(A\&B)}{A\&B\spvek{x}u\in \V}.
\end{equation}
It can be seen that $\W$ becomes a~Hilbert space with respect to the norm
\[\norm{\spvek{x}u}_{\W}^2=\norm{\spvek{x}u}_{\X\times\U}^2+\norm{A\&B\spvek{x}u}^2_{\V}.\]

\section{The dissipation inequality}\label{sec:main}

Now we are ready to introduce the concept of a dissipation inequality for input-state-output systems defined by system nodes. 
In our considerations, we restrict to semi-bounded, closed and densely defined storage functions.

\begin{definition}\label{def:dissineq}
Let $\sbvek{A\& B}{C\& D}$ be a~system node on $(\X,\U,\Y)$, assume that $Q=Q\adjun\in\Lb(\Y)$, $S\in\Lb(\U,\Y)$, $R=R\adjun\in\Lb(\U)$, and  let $s:\Y\times \U\to \R$ be as in \eqref{eq:supply_inf}. Further, let $\Sc:\X\supset\dom(\Sc)\to\R$ be a~semi-bounded, closed and densely defined real quadratic form. Then the system \eqref{eq:ODEnode} is {\em dissipative with respect to the supply rate functional $s$ and storage function $\Sc$}, if for all $T\in\R_{>0}$, any classical solution $(x(\cdot),u(\cdot),y(\cdot))$ of \eqref{eq:ODEnode} on $[0,T]$, it holds that $x(0),x(T)\in\dom(\Sc)$, and
\begin{equation}\Sc(x(T))-\Sc(x(0))+ \int_0^Ts\bigl(y(t),u(t)\bigr)\dx[t]\geq0.\label{eq:dissineq_inf}\end{equation}
Hereby, we call \eqref{eq:dissineq_inf} {\em dissipation inequality}.
\end{definition}

\begin{remark}\label{rem:dissall}
   Assume the framework of Definition~\ref{def:dissineq}, and assume that $(x(\cdot),u(\cdot),y(\cdot))$ is a~classical trajectory of \eqref{eq:ODEnode} on $[0,T]$. Then, for all $t_0,t_1\in[0,T]$ with $t_0\leq t_1$, it holds that
\[\Sc(x(t_1))-\Sc(x(t_0))+ \int_{t_0}^{t_1}s\bigl(y(t),u(t)\bigr)\dx[t]\geq0.\]
For $t_0=t_1$, this is obvious. If $t_0<t_1$, the statement follows from the definition of dissipativity together with the fact that
   $(\tilde{x}(\cdot),\tilde{u}(\cdot),\tilde{y}(\cdot))$ with $\tilde{x}(t)=x(t+t_0)$, $\tilde{u}(t)=u(t+t_0)$, $\tilde{y}(t)=y(t+t_0)$, $t\in[0,t_1-t_0]$, is a~classical trajectory of \eqref{eq:ODEnode} on $[0,t_1-t_0]$.
\end{remark}
In the sequel we collect some properties of of dissipative systems. First we show that storage functions are bounded on the space $\V$ as defined in \eqref{eq:Vdef}.
\begin{proposition}\label{prop:domS}
   Assume the framework of Definition~\ref{def:dissineq}. Then $\V$ as defined in \eqref{eq:Vdef} fulfills $\V\subset\dom(\Sc)$, and there exists a constant $c>0$ such that
    \begin{equation}\forall\,x_0\in \V:\qquad|\Sc(x_0)|\leq c\, \norm{x_0}_{\V}^2.\label{eq:SVbnd}\end{equation}
\end{proposition}
\begin{proof}
Assume that $x\in \V$. Then there exists some $u_0\in\U$ with $\spvek{x_0}{u_0}\in\dom(A\&B)$. Let $u(\cdot):[0,1]\to\U$ be the constant function $u(\cdot)\equiv u_0$. By Proposition~\ref{prop:solex} there exist $x(\cdot)\in\conC^{1}([0,1];\X)$ and $y(\cdot)\in\conC([0,1];\Y)$, such that $(x(\cdot),u(\cdot),y(\cdot))$ is a classical trajectory of \eqref{eq:ODEnode} on $[0,1]$ with $x(0)=x_0$. In particular, $x_0\in\dom(\Sc)$. We next show that $\Sc$ is bounded on $\V$: Without loss of generality, we assume that $\Sc$ is semi-bounded from below (otherwise consider $-\Sc$). Consider the embedding $\iota:\V\to\dom(\Sc)$. Closedness of $\Sc$ and the fact that $\V$ is complete implies that the graph of $\iota$ is closed. Then the closed graph theorem \cite[Thm.~7.9]{Alt16} gives $\iota\in\Lb(\V,\dom(\Sc))$. That is, for chosing $\lambda\in\R_{\ge 0}$ such that the mapping in \eqref{eq:formnorm} (with ``$\pm=+$'') defines a norm on $\dom(\Sc)$, there exists $C\ge 0$ such that
\[
\forall\,x\in\V:\quad 0\leq \Sc(x) + \lambda \norm{x}_{\X}^{2} \le C\,\norm{x}_{\V}^{2}.
\]
This immediately implies that \eqref{eq:SVbnd} holds for $c=\max\{\lambda,C\}$.
\end{proof}

The previous result can be used to derive some criteria for continuity and smoothness of the real function $t\mapsto\Sc(x(t))$, if $\Sc$ is a~storage function.
We first remark that the continuous embedding of $\dom(\Sc)$ in $\V$ implies that $\V\adjun$ is continuously embedded in $\dom(\Sc)$. Hence, $\widetilde{P}$ restricts to a~bounded operator from $\V$ to $\V\adjun$. This restriction will be also denoted by $\widetilde{P}$ for sake of simplicity.

\begin{proposition}\label{prop:Storsmooth}
   Assume the framework of Definition~\ref{def:dissineq}, let $T>0$, and let $(x(\cdot),u(\cdot),y(\cdot))$ be a~classical trajectory  of \eqref{eq:ODEnode} on $[0,T]$. Then \[\Sc(x(\cdot))\in\conC([0,T]).\]
   If, moreover,    $x(\cdot)\in\conC^{2}([0,T];\X)$ and $u(\cdot)\in\conC^{1}([0,T];\U)$, then 
\[\Sc(x(\cdot))\in\conC^1([0,T]).\]
In this case, for $\widetilde{P}\in\Lb(\dom(\Sc),\dom(\Sc)\adjun)$ such that \eqref{eq:Pdual} holds for $h$ as in \eqref{eq:hsesq}, the derivative of $\Sc(x(\cdot))$ fulfills
\[
\forall\,t\in [0,T]:\quad \ddts\mapsto \Sc(x(t))=2\Re\scprod*{\widetilde{P}x(t)}{A\&B\spvek{x(t)}{u(t)}}_{\V\adjun,\V}.
\]
\end{proposition}
\begin{proof}
If $(x(\cdot),u(\cdot),y(\cdot))$ be a~classical trajectory  of \eqref{eq:ODEnode} on $[0,T]$, then $x(\cdot)\in\conC([0,T];\V)$. Then continuity of $\Sc(x(\cdot))$ follows from $\Sc(x(t))=\scprod{Px(t)}{x(t)}_{\V\adjun,\V}$ for all $t\in[0,T]$.
If, further, $x(\cdot)\in\conC^{2}([0,T];\X)$ and $u(\cdot)\in\conC^{1}([0,T];\U)$, 
Lemma~\ref{lem:smoothsol} yields that $x(\cdot)\in \conC^{1}([0,T];\V)$. Then the product rule gives continuous differentiability of
\[t\mapsto \Sc(x(t))=\scprod*{\widetilde{P}x(t)}{x(t)}_{\V\adjun,\V}\]
with 
\begin{multline*}
\forall\,t\in [0,T]:\quad \ddts \Sc(x(t))=2\Re\scprod*{\widetilde{P}x(t)}{\dot{x}(t)}_{\V\adjun,\V}\\
=2\Re\scprod*{\widetilde{P}x(t)}{A\&B\spvek{x(t)}{u(t)}}_{\V\adjun,\V},
\end{multline*}
and the proof is complete.
\end{proof}
Now we present our main theorem on the characterization of dissipativity. 

\begin{theorem}\label{thm:main}
Let $\sbvek{A\& B}{C\& D}$ be a~system node on $(\X,\U,\Y)$, assume that $Q=Q\adjun\in\Lb(\Y)$, $S\in\Lb(\U,\Y)$, $R=R\adjun\in\Lb(\U)$, and  let $s:\Y\times \U\to \R$ be as in \eqref{eq:supply_inf}. Further, let $\Sc:\X\supset\dom(\Sc)\to\R$ be a~semi-bounded, closed and densely defined real quadratic form, let $\widetilde{P}\in\Lb(\dom(\Sc),\dom(\Sc)\adjun)$ such that  \eqref{eq:Pdual} holds for $h$ as in \eqref{eq:hsesq}). Then the following are equivalent for $\V$ and $\W$ as in \eqref{eq:Vdef} and \eqref{eq:Wdef}.
\begin{enumerate}[label=(\roman{*})]
  \item\label{item:dissineq1} The system \eqref{eq:ODEnode} is dissipative with respect to the supply rate functional $s$ and storage function $\Sc$.
  \item\label{item:dissineq2} $\V\subset\dom(\Sc)$, and there exists a dense subspace $\widetilde{\W}$ of $\W$, such that, for all $\spvek{x}u\in\widetilde{\W}$
  \begin{equation}
    \label{eq:kyp_inf_s}
2\Re\scprod*{\widetilde{P}x}{A\&B \spvek{x}u}_{\V\adjun,\V}+s\left({C\&D\spvek{x}u},u\right)\geq 0.
  \end{equation}
  \item\label{item:dissineq3}
$\V\subset\dom(\Sc)$, and \eqref{eq:kyp_inf_s}  holds for all $\spvek{x}u\in\W$.  \item\label{item:dissineq4} $\V\subset\dom(\Sc)$, and, there exists a~Hilbert space $\Z$ and some $K\&L\in \Lb(\W,\Z)$ with dense range, such that, for all $\spvek{x}u\in\W$,
    \begin{equation}
    \label{eq:kyp_inf_lure}
2\Re\scprod*{\widetilde{P}x}{A\&B \spvek{x}u}_{\V\adjun,\V}+s\left({C\&D\spvek{x}u},u\right)=\norm{K\&L\spvek{x}u}_\Z^2.
  \end{equation}
\end{enumerate}
\end{theorem}
\begin{proof}
We prove the result according to the pattern
\[
\begin{array}{cccccccccc}
  \text{\ref{item:dissineq1}}&\Rightarrow&\text{\ref{item:dissineq3}}&\Rightarrow&
\text{\ref{item:dissineq4}}&\Rightarrow&
\text{\ref{item:dissineq3}}
 \\[1mm]
    &&  \text{\rotatebox[origin=c]{90}{$\Leftrightarrow$}} &&&&  \text{\rotatebox[origin=c]{90}{$\Leftarrow$}} &  \\[1mm]&&\text{\ref{item:dissineq2}}&&&&
\text{\ref{item:dissineq1}.}
\end{array}
\]
\ref{item:dissineq1}$\Rightarrow$\ref{item:dissineq3}: Assume that \eqref{eq:ODEnode} is dissipative with respect to the supply rate functional $s$ and storage function $\Sc$. Then 
$\V\subset\dom(\Sc)$ follows from Proposition~\ref{prop:domS}. Let $\spvek{x}u\in\W$. Then Lemma~\ref{lem:solexsmooth} yields that
there exists a~classical trajectory $(x(\cdot),u(\cdot),y(\cdot))$ of \eqref{eq:ODEnode} on $[0,1]$ with $x(0)=x$, $u(0)=u$, $x(\cdot)\in \conC^{2}([0,1];\X)$ and $\spvek{x(\cdot)}{u(\cdot)}\in \conC([0,1];\W)$.
Then Proposition~\ref{prop:Storsmooth}  yields that $\scprod{\widetilde{P}x(\cdot)}{x(\cdot)}_{\V\adjun,\V}\in\conC^1([0,1])$ with
\begin{align*}
    2\Re\scprod*{\widetilde{P}x}{A\&B \spvek{x}u}_{\V\adjun,\V}&=\ddt \scprod{\widetilde{P}x(t)}{x(t)}_{\V\adjun,\V}\Big\vert_{t=0}\\&=\lim_{t\searrow0}\frac{\scprod{\widetilde{P}x(t)}{x(t)}_{\V\adjun,\V}-\scprod{\widetilde{P}x(0)}{x(0)}_{\V\adjun,\V}}{t}\\
&\geq \lim_{t\searrow0}\frac{-\int_0^t s\bigl(y(\tau),u(\tau)\bigr)}{t}\\&=-s\bigl(y(0),u(0)\bigr)=-s\left({C\&D\spvek{x}u},u\right).
\end{align*}
\ref{item:dissineq3}$\Rightarrow$\ref{item:dissineq2}: This is trivial.\\
\ref{item:dissineq2}$\Rightarrow$\ref{item:dissineq3}: This follows by a~continuity argument.\\
\ref{item:dissineq3}$\Rightarrow$\ref{item:dissineq4}: 
Consider the mapping
\begin{align*}
    \mathcal{T}:\qquad \W&\to\R,\\
    \spvek{x}{u}&\mapsto 2\Re\scprod*{\widetilde{P}x}{A\&B \spvek{x}u}_{\V\adjun,\V}+s\left({C\&D\spvek{x}u},u\right).
\end{align*}
Then $\mathcal{T}$ is a~bounded nonnegative quadratic form. Hence, by Kato's First Representation Theorem, there exists some nonnegative and self-adjoint $M\in \Lb(\W)$, such that
$\mathcal{T}(v)=\scprod{Mv}{v}_{\W}$ for all $v\in\W$. Then, by \cite[Chap.~V, Thm.~3.35]{Kato80}, $M$ possesses an operator square root $\sqrt{M}\in \Lb(\W)$, i.e., $\sqrt{M}$ is nonnegative, self-adjoint, and it fulfills $\sqrt{M}^2=M$. Now setting $\Z:=\overline{\im \sqrt{M}}$ (endowed with the inner product inherited from $\W$) and $K\&L:=\sqrt{M}$, we have that $K\&L\in \Lb(\W,\Z)$ has dense range and
\begin{align*}
\forall\,\spvek{x}{u}\in\W:\quad
&\phantom{=}2\Re\scprod*{\widetilde{P}x}{A\&B \spvek{x}u}_{\V\adjun,\V}+s\left({C\&D\spvek{x}u},u\right)\\&=\mathcal{T}(\spvek{x}{u})
=\scprod{M\spvek{x}{u}}{\spvek{x}{u}}_{\W}\\&=
\scprod*{\sqrt{M}\spvek{x}{u}}{\sqrt{M}\spvek{x}{u}}_\Z=\norm*{K\&L \spvek{x}{u}}_\Z^2.
\end{align*}
\ref{item:dissineq4}$\Rightarrow$\ref{item:dissineq3}: This follows directly from the nonnegativity of the norm.\\ 
\ref{item:dissineq3}$\Rightarrow$\ref{item:dissineq1}:
Let $(x(\cdot),u(\cdot),y(\cdot))$ be a~classical trajectory of \eqref{eq:ODEnode} on $[0,T]$. Consider a `mollifier sequence' $(\alpha_n)_{n\in\N}$ in $\conC^\infty(\R)$, see Definition~\ref{def:molli}, and consider the 
convolutions 
\begin{equation}
\begin{aligned}
x_n(\cdot)&=(\alpha_n\ast x)(\cdot):[0,T]\to\X,\\
u_n(\cdot)&=(\alpha_n\ast u)(\cdot):[0,T]\to\U, \\
y_n(\cdot)&=(\alpha_n\ast y)(\cdot) :[0,T]\to\Y.    
\end{aligned}\label{eq:conxuy}
\end{equation}
Then Lemma~\ref{lem:mollconv} yields that these functions are all infinitely often differentiable. Further, the same lemma implies that the restriction of  $(x_n(\cdot),u_n(\cdot),y_n(\cdot))$ to $[0,T-1/n]$ is a~classical trajectory of \eqref{eq:ODEnode} on $[0,T-1/n]$.
Then Lemma~\ref{lem:smoothsol}
yields that the restrictions of $x_n(\cdot)$ and $\spvek{x_n(\cdot)}{u_n(\cdot)}$ to $[0,T-1/n]$ are 
in $\conC^{1}([0,t];\V)$ and $\conC([0,t];\W)$, resp. An application of Proposition~\ref{prop:Storsmooth} leads to continuous differentiability of $\scprod{\widetilde{P}x_n(\cdot)}{x_n(\cdot)}_{\V\adjun,\V}$, and for all $t\in[0,T-1/n]$,
\[
\ddts\scprod{\widetilde{P}x_n(t)}{x_n(t)}=2\Re\scprod*{\widetilde{P}x_n(t)}{A\&B\spvek{{x}_n(t)}{u_n(t)}}_{\V\adjun,\V}\geq -s\bigl(y_n(t),u_n(t)\bigr).
\]
Now an integration over $[0,T-1/n]$ gives
\begin{equation}
\forall\,n\in\N \quad \Sc(x_n(T-1/n))-\Sc(x_n(0))+ \int_0^{T-1/n}s\bigl(y_n(t),u_n(t)\bigr)\dx[t]\geq0.\label{eq:apprdiss}
\end{equation}
It follows from the mean value theorem of integration together with uniform continuity of $x(\cdot):[0,T]\to\V$ that $(x_n(\cdot))$ converges to $x(\cdot)$ in $\V$. Further, by \cite[Thm.~4.15]{Alt16}, 
we have
\[\lim_{n\to\infty}\int_0^{T-1/n}s\bigl(y_n(t),u_n(t)\bigr)\dx[t]=\int_0^Ts\bigl(y(t),u(t)\bigr)\dx[t].\]
Now taking the limit $n\to\infty$ in \eqref{eq:apprdiss}, we obtain that the dissipation inequality holds.
\end{proof}
\begin{remark}[Dissipation inequality]\label{rem:dissineq}
Assume the framework of Theorem~\ref{thm:main}. 
\begin{enumerate}[label=(\alph{*})]
\item\label{rem:dissineq1} Consider the space
\[\W_P^0:=\setdef{\spvek{x}{u}\in \dom(A\&B)}{x\in \dom(P)},\]
which is a~Hilbert space endowed with the norm
\[\norm*{\spvek{x}{u}}_{\W_P^0}^2=\norm*{\spvek{x}{u}}_{\V}^2+\|Px\|^2_\X.\]
Under the additional assumptions that 
\begin{equation}
    \W^1_P:=\setdef{\spvek{x}{u}\in \V}{x\in \dom(P)}\label{eq:domPdense} 
\end{equation} is dense
in both $\W$ and $\W_P^0$, the statements in Theorem~\ref{thm:main} are equivalent to
\begin{enumerate}[label=(\roman{*}), start=6]
\item\label{item:dissineq6} $\W\subset\dom(\Sc)$ and 
  \begin{multline*}
\forall\,    \spvek{x}u\in\dom(A\&B)\text{ with }x\in\dom(P):\\
2\Re\scprod*{Px}{A\&B \spvek{x}u}_{\X}+s\left({C\&D\spvek{x}u},u\right)\geq 0.
  \end{multline*}
\end{enumerate}
To verify this, we first state that
\begin{equation}\label{eq:KYPdomP}
\begin{aligned}
\forall\spvek{x}{u}\in\W^1_P:\quad &\phantom{=}\;\;2\Re\scprod*{Px}{A\&B \spvek{x}u}_{\X}\quad\;\;\;+s\left({C\&D\spvek{x}u},u\right)\\&=2\Re\scprod*{\widetilde{P}x}{A\&B \spvek{x}u}_{\V\adjun,\V}+s\left({C\&D\spvek{x}u},u\right).
\end{aligned}
\end{equation}
Consequently, by using density of $\W^1_P$ in $\W$, we see that
\ref{item:dissineq6} implies \ref{item:dissineq2}. On the other hand, if \ref{item:dissineq3} holds, then \eqref{eq:KYPdomP} implies that
  \[
\forall\,    \spvek{x}u\in\W^1_P:\qquad
2\Re\scprod*{Px}{A\&B \spvek{x}u}_{\X}+s\left({C\&D\spvek{x}u},u\right)\geq 0.
  \]
Then density of $\W^1_P$ in $\W^0_P$ yields, together with a~continuity argument, that  \ref{item:dissineq6} is valid.\\
The author of this article is not certain whether the above density assumptions are always satisfied. In particular, no counterexample is known to him.
\item\label{rem:dissineq2} If \eqref{eq:ODEnode} is dissipative with respect to the supply rate functional $s$ and storage function $\Sc$,
then Theorem~\ref{thm:main} yields that there exists a~Hilbert space $\Z$ and some
$K\&L\in \Lb(\V,\Z)$ with dense range, such that \eqref{eq:kyp_inf_lure} holds for all $\spvek{x}u\in\V$.
Now using the argumentation as in Theorem~\ref{thm:main}, we obtain that any classical trajectory $(x(\cdot),u(\cdot),y(\cdot))$ of \eqref{eq:ODEnode} on $[0,T]$ fulfills
\begin{equation}
\Sc(x(T))-\Sc(x(0))+ \int_0^Ts\bigl(y(t),u(t)\bigr)\dx[t]=\int_0^T\norm*{K\&L\spvek{x(t)}{u(t)}}_\Z^2\dx[t].\label{eq:dissrate}\end{equation}
The term $\norm*{K\&L\spvek{x(t)}{u(t)}}_\Z^2$ refers to what is called {\em dissipation rate} in \cite{Will71a}.\\
Note that \eqref{eq:dissrate} implies that, for any classical trajectory $(x(\cdot),u(\cdot),y(\cdot))$ of \eqref{eq:ODEnode} on $[0,T]$, it holds that \begin{equation}
K\&L\spvek{x(\cdot)}{u(\cdot)}\in\Lp{2}([0,T];\Z).\label{eq:dissL2}
\end{equation}
\item\label{rem:dissineq3} For $k\in\N$, define
\[
H_{0r}^k([0,T];\X)
:= \setdef*{\,v \in H^k([0,T];\X)\,}{\,v(T)=\cdots=\tfrac{\mathrm d^{\,k-1}}{\mathrm dt^{\,k-1}}v(T)=0\,},
\]
and denote its antidual space by
\[
H_{0l}^{-k}([0,T];\X) := H_{0r}^k([0,T];\X)\adjun.
\]
As shown in \cite[Rem.~3.5]{FriReRe24}, generalized solutions of \eqref{eq:ODEnode} on $[0,T]$ satisfy
\[
\spvek{x(\cdot)}{u(\cdot)} \in H_{0l}^{-2}([0,T];\dom(A\&B)).
\]
By the same argument,
\[
\spvek{x(\cdot)}{u(\cdot)} \in H_{0l}^{-3}([0,T];\W).
\]
Thus, increasing the state regularity entails a decrease in temporal regularity. Consequently, for generalized solutions the expression $K\&L\spvek{x(\cdot)}{u(\cdot)}$ may be viewed as an element of $H_{0l}^{-3}([0,T];\Z)$.
\end{enumerate}
\end{remark}
We further show that the dissipation inequality holds for the class of generalized trajectories with $x(\cdot)\in\conC([0,T];\dom(\Sc))$. Moreover, we show that $K\&L\,\spvek{x(\cdot)}{u(\cdot)}$ (a distribution according to Remark~\ref{rem:dissineq}\,\ref{rem:dissineq3}) actually belongs to $\Lp{2}$ in this case.
\begin{proposition}\label{prop:domScont}
Under the assumptions of Definition~\ref{def:dissineq}, let $(x(\cdot),u(\cdot),y(\cdot))$ be a generalized trajectory of \eqref{eq:ODEnode} on $[0,T]$ with $x(\cdot)\in\conC([0,T];\dom(\Sc))$. Then the following statements hold:
\begin{enumerate}[label=(\alph{*})]
\item \eqref{eq:dissL2} and \eqref{eq:dissrate} hold for $K\&L\in\Lb(\W;\Z)$ as in Theorem~\ref{thm:main}\,\ref{item:dissineq4}.
\item $\Sc(x(\cdot))\in\Wkp{1,1}([0,T])$.
\item The dissipation inequality \eqref{eq:dissineq_inf} holds.
\end{enumerate}
\end{proposition}
\begin{proof}\
\begin{enumerate}[label=(\alph{*})]
\item Let $(x(\cdot),u(\cdot),y(\cdot))$ be a~generalized trajectory of \eqref{eq:ODEnode} on $[0,T]$. Consider a `mollifier sequence' $(\alpha_n)_{n\in\N}$ in $\conC^\infty(\R)$, see Definition~\ref{def:molli}, and consider the 
convolutions $x_n(\cdot)$, $u_n(\cdot)$ and $y_n(\cdot)$ as in \eqref{eq:conxuy}.
Then Lemma~\ref{lem:mollconv} yields that, for all $n\in\N$, the restriction of
\[\big(x_n(\cdot),u_n(\cdot),y_n(\cdot)\big)\]
to $[0,T-1/n]$ is a~smooth classical trajectory of \eqref{eq:ODEnode} on $[0,T-1/n]$. Hence, by Remark~\ref{rem:dissineq}\,\ref{rem:dissineq2}, we have, for all $n\in\N$, $T_f\in[0,T-1/n]$,
\[\Sc(x_n(T_f))-\Sc(x_n(0))+ \int_0^{T_f}s\bigl(y_n(t),u_n(t)\bigr)\dx[t]=\int_0^{T-1/n}\norm*{K\&L\spvek{x_n(t)}{u_n(t)}}_\Z^2\dx[t].\]
Now taking the limit $n\to\infty$, we obtain from Lemma~\ref{lem:mollconv} (with $\widetilde{X}=\dom(\Sc)$) that the right hand side of the above equation converges. In particular, the restriction of $K\&L\spvek{x(\cdot)}{u(\cdot)}$ to $[0,T_f)$ is in $\Lp{2}$ with
\begin{multline*}
\forall\,T_f\in[0,T):\quad \Sc(x(T_f))-\Sc(x(0))+ \int_0^{T_f}s\bigl(y(t),u(t)\bigr)\dx[t]\\=\norm*{K\&L\spvek{x(\cdot)}{u(\cdot)}}_{\Lp{2}([0,T_f];\Z)}^2.
\end{multline*}
Invoking that $x(\cdot)$ is continuous on $[0,T]$ as a~mapping to $\dom(\Sc)$, the limit $T_f\nearrow T$ leads to \eqref{eq:dissL2} and \eqref{eq:dissrate}.
\item 
By using the already proven part, we obtain that, for
\[z(\cdot):=K\&L\spvek{x(\cdot)}{u(\cdot)},\]
it holds that
\[\forall\, T_f\in[0,T]:\quad \Sc(x(T_f))=\Sc(x(0))- \int_0^{T_f}s\bigl(y(t),u(t)\bigr)\dx[t]+\int_0^{T_f}\norm{z(t)}_{\Z}^2\dx[t].\]
The right hand side is, because of square integrability of $u(\cdot)$, $y(\cdot)$ and $z(\cdot)$ weakly differentiable with integrable derivative (equivalently, it is absolutely continuous), i.e.,
$\Sc(x(\cdot))\in\Wkp{1,1}([0,T])$.
\item This is a~direct consequence of the first statement in this proposition.
\end{enumerate}
\end{proof}

We make a few additional comments on passive systems.
\begin{remark}[Passive systems]
In \cite{Staf02,ArSt07,St02}, so-called \emph{impedance passive} and \emph{scattering passive} systems are considered. In the notation of \eqref{eq:supply_inf}, these correspond to the supply rate functionals \eqref{eq:supply_inf} with
\[
\begin{aligned}
\text{impedance passive:}\quad 
& \sbmat{Q}{S}{S\adjun}{R} \;=\; \sbmat{0}{\Rf_\U\adjun}{\Rf_\U}{0},\quad\text{with $\Y=\U\adjun$}\\[0.5ex]
\text{scattering passive:}\quad 
& \sbmat{Q}{S}{S\adjun}{R} \;=\; \,\sbmat{-\id_\Y}{0}{0}{\id_\U}.
\end{aligned}
\]
Hereby, $\Rf_\U$ is the {\em Riesz map}, sending $u\in \U$ to the functional $v\mapsto\langle u,v\rangle_{\U}$. Since $\U\adjun$ is the anti-dual of $\U$, we have that $\Rf_\U$ is a linear isometric isomorphism from $\U$ to $\U\adjun$. In particular, $\Rf_\U\adjun=\Rf_\U^{-1}$.

The more special concepts of {\em internal impedance/scattering passivity} have been considered, which refer to the dissipation inequality
\[\norm{x(T)}^2-\norm{x(0)}^2\leq \int_0^ts\bigl(y(t),u(t)\bigr)\dx[t].\]
with respective supply rate functionals as above. 
In the context of the dissipation inequality \eqref{eq:dissineq_inf}, this gives rise to the storage function $\Sc(x)=-\norm{x}_\X^2$.  This property has been characterized by means of the so called {\em Kalman–Yakubovich–Popov inequality}, which basically means that \
\[\forall\, \spvek{x}u\in\dom(A\&B):\quad 2\Re\scprod*{x}{A\&B \spvek{x}u}_{\V\adjun,\V}-s\left({C\&D\spvek{x}u},u\right)\leq 0.\]
Apart from the specialization to specific supply rate functionals, this is related to \eqref{eq:kyp_inf_s} via the consideration of $P=-\id_\X$, that the inequality holds on $\dom(A\&B)$ instead of the smaller space $\W$ and that the inner product on $\X$ is considered instead of the duality product on $\V$ and its anti-dual. The latter can be inferred also from \eqref{eq:kyp_inf_s} by combining boundedness of $P$ with a~density argument. Here we also refer to the forthcoming section, in which the additional benefit of bounded storage functions is highlighted.

Furthermore, \cite{ArSt07} studies, for supply functions corresponding to impedance and scattering passivity, systems that  satisfy $\int_0^Ts\bigl(y(t),u(t)\bigr)\dx[t]\geq0$ 
for all $T>0$ and for every classical (hence also every generalized) trajectory on $[0,T]$ with $x(0)=0$. Under the assumptions of approximate controllability and approximate observability, it is shown that nonpositive storage functions exist. This indicates that restricting attention to bounded (or coercive) storage functions is overly restrictive. In the impedance passive case, the existence of bounded storage functions requires additional hypotheses, such as $A$ generating a strongly continuous group on $\X$ \cite{Pand99}. Note that the results 
\cite{ArSt07,Pand99} generalize the so-called {\em positive real lemma} and {\em bounded real lemma} (see \cite{AndeV73}) to the infinite-dimensional case.
\end{remark}

\section{Bounded storage functions}\label{sec:bndstor}
We now derive additional properties that occur in the special case where the storage function is bounded. In particular, we establish a variant of our main result (Theorem~\ref{thm:main}) under the additional assumption of a bounded storage function.

It is almost superfluous to note that boundedness, together with closedness and dense domain, implies that a form is defined everywhere. 
Further, the operators induced by such forms are bounded and defined on all of $\X$.

\begin{theorem}\label{thm:mainbnd}
Let $\sbvek{A\& B}{C\& D}$ be a~system node on $(\X,\U,\Y)$, assume that $Q=Q\adjun\in\Lb(\Y)$, $S\in\Lb(\U,\Y)$, $R=R\adjun\in\Lb(\U)$, and  let $s:\Y\times \U\to \R$ be as in \eqref{eq:supply_inf}. Further, let $\Sc:\X\to\R$ be a~bounded and real quadratic form, and let $P\in\Lb(\X)$ be the operator induced by $\Sc$. Then the following are equivalent for $\V$ and $\W$ as in \eqref{eq:Vdef} and \eqref{eq:Wdef}.
\begin{enumerate}[label=(\roman{*}')]
  \item\label{item:dissineq1bnd} The system \eqref{eq:ODEnode} is dissipative with respect to the supply rate functional $s$ and storage function $\Sc$.
  \item\label{item:dissineq2bnd} There exists a dense subspace $\widetilde{\W}$ of $\dom(A\&B)$, such that, for all $\spvek{x}u\in\widetilde{\W}$
  \begin{equation}
    \label{eq:kyp_inf_s_bnd}
2\Re\scprod*{Px}{A\&B \spvek{x}u}_{\X}+s\left({C\&D\spvek{x}u},u\right)\geq 0.
  \end{equation}
  \item\label{item:dissineq3bnd}
\eqref{eq:kyp_inf_s_bnd}  holds for all $\spvek{x}u\in\dom(A\&B)$.
  \item\label{item:dissineq4bnd} There exists a~Hilbert space $\Z$ and some $K\&L\in \Lb(\dom(A\&B),\Z)$ with dense range, such that, for all $\spvek{x}u\in\dom(A\&B)$,
    \begin{equation}
    \label{eq:kyp_inf_lurebnd}
2\Re\scprod*{{P}x}{A\&B \spvek{x}u}_{\X}+s\left({C\&D\spvek{x}u},u\right)=\norm{K\&L\spvek{x}u}_\Z^2.
  \end{equation}
\item\label{item:dissineq5bnd} The dissipation inequality \eqref{eq:dissineq_inf} holds for all generalized trajectories $(x(\cdot),u(\cdot),y(\cdot))$ of \eqref{eq:ODEnode} on $[0,T]$.
\end{enumerate}
\end{theorem}
\begin{proof}
In the following, the lower-case Roman numerals without primes refer to the corresponding statements in Theorem~\ref{thm:main}. Our present result is established once we prove the following implications under the additional assumption of bounded storage function:
\[
\begin{array}{cccccccc}
\text{\ref{item:dissineq1}}&\Leftrightarrow&\text{\ref{item:dissineq1bnd}}&\Leftrightarrow&\text{\ref{item:dissineq5bnd}}\\
\text{\ref{item:dissineq3}}&\Rightarrow&\text{\ref{item:dissineq2bnd}}&\Rightarrow&\text{\ref{item:dissineq3bnd}}&\Rightarrow&\text{\ref{item:dissineq3}}\\
\text{\ref{item:dissineq4}}&\Leftrightarrow&\text{\ref{item:dissineq4bnd}.}
\end{array}\]
Let us first show the first line: The equivalence between \ref{item:dissineq1bnd} and \ref{item:dissineq1} is immediate, as they state exactly the same assertion. Since classical trajectories are generalized trajectories, we further have that \ref{item:dissineq5bnd} implies \ref{item:dissineq1bnd}. Since, our assumption on $\Sc$ gives $\dom(\Sc)=\X$, the implication ``\ref{item:dissineq1bnd}$\Rightarrow$\ref{item:dissineq5bnd}'' follows from Proposition~\ref{prop:domScont}.\\
For the remaining parts, we first observe that the boundedness of $\Sc$ implies that $\widetilde{P}$ from Theorem~\ref{thm:main} belongs to $\Lb(\X)$. Consequently, we have $\widetilde{P}=P$, a fact that will be used in the sequel.\\
\ref{item:dissineq3}$\Rightarrow$\ref{item:dissineq2bnd}: This follows from the fact that $\W$ is a core of $A\&B$, see Lemma~\ref{eq:Vcore}.\\
\ref{item:dissineq2bnd}$\Rightarrow$\ref{item:dissineq3bnd}: A straightforward continuity argument yields this implication.\\
\ref{item:dissineq3bnd}$\Rightarrow$\ref{item:dissineq3}:
This follows from $\W\subset\dom(A\&B)$.\\
It remains to prove the equivalence between \ref{item:dissineq4} and \ref{item:dissineq4bnd}. Since $\W\subset\dom(A\&B)$, we have that \ref{item:dissineq4bnd} implies \ref{item:dissineq4}. On the other hand, assume that \ref{item:dissineq4} holds. Then, by using $P=\widetilde{P}\in\Lb(\X)$,
there exists a~Hilbert space $\Z$ and some $K\&L\in \Lb(\V,\Z)$ with dense range, such that, for all $\spvek{x}u\in\W$,
\begin{equation}
    \label{eq:Lurebnd}
    2\Re\scprod*{{P}x}{A\&B \spvek{x}u}_{\X}+s\left({C\&D\spvek{x}u},u\right)=\norm{K\&L\spvek{x}u}_\Z^2.
    \end{equation}
Then boundedness of $P$ yields that there exists some $M>0$, such that, for all $\spvek{x}u\in\W$,
\[\norm{K\&L\spvek{x}u}_\Z^2\leq \norm{\spvek{x}u}_{\dom(A\&B)}^2.\]
The density of $\W$ in $\dom(A\&B)$ (see Lemma~\ref{eq:Vcore}) implies that $K\&L$ extends to a bounded operator from $\dom(A\&B)$ to $\Z$. Clearly, its range is dense as well. Then a~continuity argument yields that \eqref{eq:Lurebnd} holds for all $\spvek{x}u\in\dom(A\&B)$. Thus, \ref{item:dissineq4bnd} holds, and the proof is finished.
\end{proof}
\begin{remark}[Bounded storage function]
Assume the framework of Theorem~\ref{thm:mainbnd}. If \eqref{eq:ODEnode}
is dissipative with respect to the supply rate functional $s$ and the bounded storage function $\Sc:\X\to\R$, then, as shown in Theorem~\ref{thm:mainbnd}, the operator $K\&L$ in \ref{item:dissineq4bnd} is bounded from $\dom(A\&B)$ to $\Z$. This yields that
\begin{equation}\left[\begin{smallmatrix}A\&B\\[1mm]C\&D\\[0.5mm]K\&L\end{smallmatrix}\right]\label{eq:extsystnode}\end{equation}
is a~system node on $(\X,\U,\Y\times\Z)$.
\end{remark}
A comparison between Theorem~\ref{thm:main} and Theorem~\ref{thm:mainbnd} shows that the situation simplifies drastically when the storage function is bounded. By contrast, the additional benefit of well-posedness of the system (i.e., the existence of a constant $c_T$ for every $T$ such that all classical trajectories satisfy \eqref{eq:wp}) is comparatively minor. A small apparent advantage of well-posedness is that, by standard density and continuity arguments, it guarantees the existence of a generalized trajectory $(x(\cdot),u(\cdot),y(\cdot))$ of \eqref{eq:ODEnode} on $[0,T]$ for every $x_0\in\X$ and every $u(\cdot)\in \Lp{2}([0,T];\U)$. In the case of a bounded storage function, such a trajectory satisfies the dissipation inequality, as established in Theorem~\ref{thm:mainbnd}. 

Finally in this section, we show that Proposition~\ref{prop:domScont} likewise simplifies in the case of bounded storage functions. The validity of the third statement in Proposition~\ref{prop:domScont} need not be restated, as it is already included in Theorem~\ref{thm:mainbnd}. Since boundedness of the form implies $\dom(\Sc)=\X$, the following statement is an immediate consequence of Proposition~\ref{prop:domScont}, and we therefore omit a detailed proof.

\begin{proposition}\label{prop:domScontbnd}
Assume the hypotheses of Theorem~\ref{thm:mainbnd}, and in addition that \eqref{eq:ODEnode} is dissipative with respect to the supply rate functional $s$ with a bounded storage function $\Sc:\X\to\R$. Then, for every generalized trajectory $(x(\cdot),u(\cdot),y(\cdot))$ of \eqref{eq:ODEnode} on $[0,T]$, the following statements hold:
\begin{enumerate}[label=(\alph{*})]
\item \eqref{eq:dissL2} and \eqref{eq:dissrate} hold for $K\&L\in\Lb(\dom(A\&B);\Z)$ as in Theorem~\ref{thm:mainbnd}\,\ref{item:dissineq4bnd}.
\item $\Sc(x(\cdot))\in\Wkp{1,1}([0,T])$.
\end{enumerate}
\end{proposition}


\section{Linear-quadratic optimal control}\label{sec:lqr}
Given 
a~system node $\sbvek{A\& B}{C\& D}$ on $(\X,\U,\Y)$, and, for $Q=Q\adjun\in\Lb(\Y)$, $S\in\Lb(\U,\Y)$, $R=R\adjun\in\Lb(\U)$, let the functional $s:\Y\times \U\to \R$ be as in \eqref{eq:supply_inf}. Loosely speaking, we consider the optimal control problem for $x_0\in\X$:
\begin{multline}\label{eq:optcontr}
\text{Minimize }\;\mathcal{J}\bigl(y(\cdot),u(\cdot)\bigr):=\int_0^\infty s\bigl(y(t),u(t)\bigr)\dx[t]\\
\text{subject to }\;\spvek{\dot{x}(t)}{y(t)}= \sbvek{A\&B\\[-1mm]}{C\&D}\,\spvek{x(t)}{u(t)},\;x(0)=x_0.
\end{multline}
We study the optimal control problem under a set of additional assumptions that effectively restrict attention to bounded storage functions. There is certainly room to relax these hypotheses; however, the main focus of this article is the dissipation inequality, and we therefore do not treat the optimal control problem in its most general form.

Besides a stabilizability assumption on the system (introduced and discussed in due course), we restrict ourselves to nonnegative cost functionals. That is, we assume that
\begin{equation}
    \forall \,u\in\U,\,y\in\Y:\quad s\bigl(y,u\bigr)\geq0.\label{eq:spos}
\end{equation}
Note that the latter is equivalent to nonnegativity of the operator $\sbmat{W}{S}{S\adjun}{R}\in\Lb(\Y\times\U)$.\\

In fact, our main interest lies not in the optimal control itself, but rather in the 
so-called \emph{value function}, which is defined by $\mathrm{Val}:\X\to [0,\infty]$ with
\begin{equation}
    \mathrm{Val}(x_0) =
\inf\setdef*{\mathcal{J}\bigl(y(\cdot),u(\cdot)\bigr)\dx[t]}{
    \parbox{0.44\linewidth}{
      $(x(\cdot),u(\cdot),y(\cdot))$ is a gen.\ trajec-\\[0.3em]tory of \eqref{eq:ODEnode} on $\R_{\ge0}$
      with $x(0)=x_0$
    }
  }.\label{eq:valfun}
\end{equation}
Our goal here is to show that the value function constitutes a storage function with respect to the supply rate functional $s$.
We restrict our attention to optimal control problem under the further assumption of {\em state-feedback stabilizability}. 
This requires the notion of \emph{admissible state feedback} (see \cite[Def.~7.3.2]{Staf05} for details), which is only vaguely explained here. 
An admissible state feedback for \eqref{eq:ODEnode} is an operator $F\&G\in\Lb(\dom(A\&B);\U)$ such that, after appending it as an extra output
\begin{equation}
\left(\begin{smallmatrix}\dot{x}(t)\\y(t)\\z(t)\end{smallmatrix}\right)=\left[\begin{smallmatrix}A\&B\\[1mm]C\&D\\[1mm]F\&G\end{smallmatrix}\right]\spvek{x(t)}{u(t)},\label{eq:extsysnode}
\end{equation}
the feedback $u(\cdot)=z(\cdot)+v(\cdot)$ yields again a~system that is described by a~system node (with input $v(\cdot)$ and output $\spvek{y(\cdot)}{z(\cdot)}$). 
In words: the feedback reads $(x,u)$ through $F\&G$ and feeds it back into the input. In particular, criteria for admissibility in terms of the system node are presented in \cite[Sec.~7.3]{Staf05}. In particular, admissibility of a~state feedback is related to invertibility of the operator
\[M=\left[\begin{smallmatrix}\id_\X&0\\0&\id_\U\end{smallmatrix}\right]-\left[\begin{smallmatrix}0\\F\&G\end{smallmatrix}\right]\]
as a~mapping from $\dom(A\&B)$ to $\dom(A_K\&B_K)$, and
\[\left[\begin{smallmatrix}A_K\&B_K\\[1mm]C_K\&D_K\\[1mm]F_K\&G_K\end{smallmatrix}\right]:=\left[\begin{smallmatrix}A\&B\\[1mm]C\&D\\[1mm]F\&G\end{smallmatrix}\right]M^{-1}\]
is a~system node (the one of the closed-loop system). The trajectories $(x(\cdot),v(\cdot),y(\cdot))$ of the system defined by the latter system node are related to those of \eqref{eq:extsysnode} via  $u(\cdot)=z(\cdot)+v(\cdot)$; in particular, the state and output trajectories coincide in both descriptions.

\begin{definition}[State-feedback stabilizable]\label{def:stabilizable}
    Let $\sbvek{A\& B}{C\& D}$ be a~system node on $(\X,\U,\Y)$. 
    Then the system \eqref{eq:ODEnode} is said to be {\em state-feedback stabilizable}, if there exists an admissible state feedback $F\&G\in\Lb(\dom(A\&B))$, such that the resulting feedback system \eqref{eq:extsysnode}, $u(\cdot)=z(\cdot)+v(\cdot)$, is infinite-time well-posed.
\end{definition}
Let us briefly discuss the finite input-output cost condition in more detail.
\begin{lemma}\label{lem:iobnd}
Let $\sbvek{A\& B}{C\& D}$ be a~system node on $(\X,\U,\Y)$, such that the system \eqref{eq:ODEnode} is state-feedback stabilizable.
Then, for all $x_0\in\X$ there exists a~generalized trajectory 
$(x(\cdot),u(\cdot),y(\cdot))$ of \eqref{eq:ODEnode} on $\R_{\ge0}$ with $x(0)=x_0$, $u(\cdot)\in\Lp{2}(\R_{\ge0};\U)$ and $y(\cdot)\in\Lp{2}(\R_{\ge0};\Y)$.
Further, there exists some $M>0$ with the following property: 
\begin{multline*}
    \forall\,x_0\in\X\text{ with }\norm{x_0}_\X\leq1:\\
    \exists\, \text{gen.\ trajectory 
$(x(\cdot),u(\cdot),y(\cdot))$ of \eqref{eq:ODEnode} on $\R_{\ge0}$}\\\text{ with }
x(0)=x_0\,\wedge\,\norm*{\spvek{y(\cdot)}{u(\cdot)}}_{\Lp{2}(\R_{\ge0};\Y\times\U)}\leq M.
\end{multline*}
\end{lemma}
\begin{proof}
Consider an infinite-time stabilizing state feedback $F\&G\in\Lb(\dom(A\&B))$. Let $x_0\in\X$. Now driving the closed-loop system
\eqref{eq:extsysnode}, $u(\cdot)=z(\cdot)+v(\cdot)$, with $v(\cdot)=0$,
we obtain that there exists a~solution \[\left(x(\cdot),0,\spvek{y(\cdot)}{z(\cdot)}\right)\text{ with $x(0)=x_0$ and $\spvek{y(\cdot)}{z(\cdot)}\in\Lp{2}(\R_{\ge0};\Y\times\U)$}.\] 
 Since
$(x(\cdot),{z(\cdot)},{y(\cdot)})$ is then a~generalized solution of \eqref{eq:ODEnode} on $\R_{\ge0}$, the first statement holds. Further, infinite-time well-posedness of the closed-loop yields that the closed-loop system (again driven by $v(\cdot)$) has the property that, for all $x_0\in\X$, the output fulfills
\[\norm*{\spvek{y(\cdot)}{z(\cdot)}}_{\Lp{2}(\R_{\ge0};\Y\times\U)}\leq M \norm{x_0}_\X.\]
Now invoking that, by $v(\cdot)=0$, we have $z(\cdot)=u(\cdot)$, we obtain that the desired estimate holds.
\end{proof}

Combining Lemma~\ref{lem:iobnd} with the boundedness of $Q$, $S$, and $R$ in the cost functional $\mathcal{J}$ yields that the value function is finite for every $x_0\in\X$, i.e., 
$\mathrm{Val}(x_0)<\infty$ for all $x_0\in\X$.

Before presenting our main result on the connection between value functions and storage functions, we briefly discuss the assumption of state-feedback stabilizability.

\begin{remark}[State-feedback stabilizable]\
    \begin{enumerate}[label=(\alph{*})]
        \item Our notion in Definition~\ref{def:stabilizable} differs slightly from the classical
finite-dimensional one (see, e.g., \cite{TrentelmanStoorvogelHautus2001}), where
stabilizability means that there exists a (static) state feedback $u(t)=Fx(t)+v(t)$
such that the closed-loop matrix $A+BF$ is Hurwitz (equivalently, all unstable modes
are controllable).

By contrast, Definition~\ref{def:stabilizable} employs an \emph{admissible} state
feedback of the form $F\&G$ and requires that the closed-loop interconnection defines
an \emph{infinite-time well-posed} system node. In particular, this notion permits semisimple modes on the imaginary axis that are simultaneously uncontrollable and unobservable. Such modes are excluded by the standard finite-dimensional definition.
\item If the system is internally passive, then the output feedback $u(t)=-y(t)+v(t)$ renders the closed-loop system infinite-time well-posed; see \cite[Thm.~4.2]{HastirPaunonen2025}. In particular, every internally passive system is state-feedback stabilizable (take the admissible feedback $F\&G=-\,(C\&D)$ so that $u(\cdot)=-y(\cdot)+v(\cdot)$).

As a consequence, the class of port-Hamiltonian systems governed by system nodes (as treated in \cite{PhilippReisSchaller2025}) is state-feedback stabilizable. 
    \end{enumerate}
\end{remark}

Next we present our main result on the optimal control problem \eqref{eq:optcontr}.
\begin{theorem}\label{thm:costfun}
   Let $\sbvek{A\& B}{C\& D}$ be a~system node on $(\X,\U,\Y)$, let $Q=Q\adjun\in\Lb(\Y)$, $S\in\Lb(\U,\Y)$, $R=R\adjun\in\Lb(\U)$. Further assume that $s:\Y\times \U\to \R$ as in \eqref{eq:supply_inf} fulfills \eqref{eq:spos}, and the system \eqref{eq:ODEnode} is state-feedback stabilizable. 
   Then the following holds:
\begin{enumerate}[label=(\alph{*})]
\item The value function $\mathrm{Val}:\X\to\R$ is a~nonnegative and bounded quadratic form.
\item \eqref{eq:ODEnode} is dissipative with respect to the supply rate functional $s$ and storage function $\mathrm{Val}$. 
\item For any further nonnegative $\Sc:\X\supset\dom(\Sc)\to\R$, such that \eqref{eq:ODEnode} is dissipative with respect to the supply rate functional $s$ and storage function $\Sc$, it holds that $\Sc$ is bounded (and thus $\dom(\Sc)=\X$), and
\[\forall\,x_0\in\X:\quad \Sc(x_0)\leq \mathrm{Val}(x_0).\]
\end{enumerate}
\end{theorem}
\begin{proof}\
For brevity, we adopt the following convention: a ``generalized trajectory'' means a generalized trajectory of \eqref{eq:ODEnode} on $\R_{\ge0}$.
\begin{enumerate}[label=(\alph{*})]
\item Nonnegativity of $\Sc$ follows directly from nonnegativity of $s(\cdot,\cdot)$. Further, it follows from Lemma~\ref{lem:iobnd} that $\mathrm{Val}(x_0)\in\R$ for all  $x_0\in\X$.\\
The statement that $\Sc$ is a~bounded quadratic form  will be proven in several steps.\\
{\em Step~1}: We show that $\mathrm{Val}(\lambda x_0)=|\lambda|^2\mathrm{Val}(x_0)$ for all $x_0\in\X$, $\lambda\in\C$. The statement is clear for $\lambda=0$, so we can assume that $\lambda\neq0$ in the following.\\
Let $x_0\in\X$ and $\varepsilon>0$. Then, by Lemma~\ref{lem:iobnd} and the definition of the value function, there exists a~generalized trajectory $(x(\cdot),u(\cdot),y(\cdot))$ with $x(0)=x_0$ and 
\[\mathcal{J}\bigl(y(\cdot),u(\cdot)\bigr)\leq \mathrm{Val}(x_0)+\frac{\varepsilon}{|\lambda|^2}.\]
Using that $(\lambda x(\cdot),\lambda u(\cdot),\lambda y(\cdot))$ is a~generalized trajectory with $\lambda x(0)=\lambda x_0$, we obtain, by 
quadraticity of $s$,
\begin{multline*}
 \mathrm{Val}(\lambda x_0)\leq \mathcal{J}(\lambda x(\cdot),\lambda u(\cdot))=|\lambda|^2\mathcal{J}( y(\cdot),u(\cdot))
  = |\lambda|^2\,\mathcal{J}(y(\cdot),u(\cdot))\\
  \leq |\lambda|^2\left(\mathrm{Val}(x_0)+\frac{\varepsilon}{|\lambda|^2}\right)
  =|\lambda|^2\mathrm{Val}(x_0)+\varepsilon.
\end{multline*}
Since the latter holds for all $\varepsilon>0$, we have
$\mathrm{Val}(\lambda x_0)\leq |\lambda|^2 \Sc(x_0)$.
Applying this with $x_0$ replaced by $\lambda x_0$ and with the factor $1/\lambda$ gives the reverse inequality, hence $\mathrm{Val}(\lambda x_0)=|\lambda|^2\mathrm{Val}(x_0)$.\\
{\em Step~2}: We show the parallelogram identity
\[
  \forall\,x_{0,1},x_{0,2}\in\X:\quad \mathrm{Val}(x_{0,1}+x_{0,2})+\mathrm{Val}(x_{0,1}-x_{0,2})
  =2\mathrm{Val}(x_{0,1})+2\mathrm{Val}(x_{0,2})
  .
\]
Let $x_{0,1},x_{0,2}\in\X$ and $\varepsilon>0$. Then there exist 
generalized trajectories $(x_i(\cdot),u_i(\cdot),y_i(\cdot))$ with $x_i(0)=x_{0,i}$ such that
\[
\mathcal{J}\bigl(y_i(\cdot),u_i(\cdot)\bigr)\le\mathrm{Val}(x_{0,i})+\frac{\varepsilon}{4},
  \qquad i=1,2.
\]
Clearly, the sums and differences
\[
  (x_\pm(\cdot),u_\pm(\cdot),y_\pm(\cdot)):=\bigl(x_1(\cdot)\pm x_2(\cdot),\;u_1(\cdot)\pm u_2(\cdot),\;y_1(\cdot)\pm y_2(\cdot)\bigr)
\]
are generalized trajectories with initial states $x_{0,1}\pm x_{0,2}$. Since $s$ is quadratic,
we have the integrand-level parallelogram identity, and hence
\[
  \mathcal{J}\bigl(y_+,u_+\bigr)+\mathcal{J}\bigl(y_-,u_-\bigr)
=2\mathcal{J}\bigl(y_1,u_1\bigr)+2\mathcal{J}\bigl(y_2,u_2\bigr).
\]
Therefore,
\[
  \mathrm{Val}(x_{0,1}+x_{0,2})+\mathrm{Val}(x_{0,1}-x_{0,2})
  \le2\,\mathrm{Val}(x_{0,1})+2\mathrm{Val}(x_{0,2})+\varepsilon.
\]
Since $\varepsilon>0$ is arbitrary, we obtain
\[
  \mathrm{Val}(x_{0,1}+x_{0,2})+\mathrm{Val}(x_{0,1}-x_{0,2})
  \le2\,\mathrm{Val}(x_{0,1})+2\mathrm{Val}(x_{0,2}).
\]
For the reverse inequality, set $\tilde x_{0,1}:=\tfrac12(x_{0,1}+x_{0,2})$ and 
$\tilde x_{0,2}:=\tfrac12(x_{0,1}-x_{0,2})$. Applying the previous estimate to 
$\tilde x_{0,1},\tilde x_{0,2}$ and using Step~1 with $\lambda=\tfrac12$, we get
\begin{multline*}
  2\mathrm{Val}(x_{0,1})+2\mathrm{Val}(x_{0,2})
  =4\mathrm{Val}(\tilde x_{0,1})+4\,\mathrm{Val}(\tilde x_{0,2})\\
  \le\mathrm{Val}(x_{0,1}+x_{0,2})+\mathrm{Val}(x_{0,1}-x_{0,2}),
\end{multline*}
which proves the parallelogram identity.\\
{\em Step~3:} We conclude that $\mathrm{Val}$ is a quadratic form. Indeed, Steps~1-2 verify the hypotheses of \cite[Thm.~1]{Kurepa1965QSF}; hence, by that theorem, \eqref{eq:hsesq} is a~sesquilinear form. Thus $\mathrm{Val}:\X\to\R$ is a quadratic form.\\
{\em Step~4:} We show that $\mathrm{Val}$ is a bounded quadratic form. It can be concluded from  Lemma~\ref{lem:iobnd} that, for $M$ as in that lemma, 
\[\forall\, x_0\in\X\text{ with }\norm{x_0}_\X\leq 1:\quad\mathrm{Val}(x_0)\leq M\,\norm*{\sbmat{Q}{S}{S\adjun}{R}}.\] Hence, by using that $\mathrm{Val}$ is quadratic, we have \[\forall\, x_0\in\X:\quad\mathrm{Val}(x_0)\leq M\,\norm*{\sbmat{Q}{S}{S\adjun}{R}}\,\norm{x_0}_\X^2.\]
\item 
Assume that $(x(\cdot),u(\cdot),y(\cdot))$ is a~classical trajectory on $[0,T]$, $T>0$. 
We essentially follow the argument of the corresponding finite-dimensional proof in \cite{Will71a}. 
For $\varepsilon>0$, let $({x}_1(\cdot),{u}_1(\cdot),{y}_1(\cdot))$
be a generalized solution with ${x}_1(0)=x(T)$ and 
\[
\mathcal{J}\bigl({y}_1(\cdot),{u}_1(\cdot)\bigr)\le\mathrm{Val}(x(T))+{\varepsilon}.
\]
Now consider the \emph{concatenation} of the two trajectories, defined by
\[
(\tilde x(t),\tilde u(t),\tilde y(t))=
\begin{cases}
(x(t),u(t),y(t)), & 0\le t\le T,\\
\bigl(x_1(t-T),u_1(t-T),y_1(t-T)\bigr), & t> T,
\end{cases}
\]
Then $(\tilde x(\cdot),\tilde u(\cdot),\tilde y(\cdot))$ is a~generalized trajectory with $\tilde u(\cdot)\in\Lp{2}(\R_{\ge0};\U)$ and $\tilde y(\cdot)\in\Lp{2}(\R_{\ge0};\Y)$, whence
\begin{align*}
    \mathrm{Val}(x(0))&=\mathrm{Val}(\tilde{x}(0))\leq \mathcal{J}(\tilde{y}(\cdot),\tilde{u}(\cdot))\\&=\int_0^Ts\bigl(\tilde{y}(t),\tilde{u}(t)\bigr)\dx[t]+\int_0^\infty s\bigl(\tilde{y}(t+T),\tilde{u}(t+T)\bigr)\dx[t]\\
        &=\int_0^Ts\bigl({y}(t),{u}(t)\bigr)\dx[t]+\mathcal{J}\bigl({y}_1(\cdot),{u}_1(\cdot)\bigr)\dx[t]\\
                &\leq\int_0^Ts\bigl({y}(t),{u}(t)\bigr)\dx[t]+\mathrm{Val}(x(T))+\varepsilon.
\end{align*}
Since this holds for all $\varepsilon>0$, we can conclude that the dissipation inequality holds.
\item Let $x_0\in\V$, with $\V$ as in \eqref{eq:Vdef}. For every classical trajectory
$(x(\cdot),u(\cdot),y(\cdot))$ on $\R_{\ge0}$ with $x(0)=x_0$ we have, for all $T>0$,
\[
  0 \;\le\; \Sc(x(T))-\Sc(x_0) \;+\; \int_0^T s\bigl(y(t),u(t)\bigr)\,\dx[t].
\]
Since $\Sc\ge0$, letting $T\to\infty$, we obtain
\[
  \Sc(x_0)\;\le\;\int_0^\infty s\bigl(y(t),u(t)\bigr)\,\dx[t].
\]
Taking the infimum over all trajectories with $x(0)=x_0$ yields
\[
  \Sc(x_0)\;\le\;\mathrm{Val}(x_0).
\]
Thus $\Sc$ is bounded above by $\mathrm{Val}$ on $\V$. By density of $\V$ in $\X$
(and nonnegativity of $\Sc$), the inequality in fact holds for all
$x_0\in\X$, which completes the proof.
\end{enumerate}
\end{proof}
Next we show that, along any generalized trajectory whose input and output are square-integrable, $\mathrm{Val}(x(t))$ tends to zero for $t\to\infty$.
\begin{proposition}\label{prop:xtends0}
Assume the standing hypotheses of Theorem~\ref{thm:costfun}, and let
$(x(\cdot),u(\cdot),y(\cdot))$ be a generalized trajectory of \eqref{eq:ODEnode} on $\R_{\ge0}$
with $x(0)=x_0$, $u(\cdot)\in\Lp{2}(\R_{\ge0};\U)$, and $y(\cdot)\in\Lp{2}(\R_{\ge0};\Y)$. Then
\[
  \lim_{t\to\infty} \mathrm{Val}\bigl(x(t)\bigr)=0.
\]
\end{proposition}
\begin{proof}
   By a~simple shift, we see that 
    \[\forall\, t\in\R_{\ge0}:\quad 0\leq \mathrm{Val}(x(t))\leq \int_t^\infty s(y(\tau),u(\tau))\dx[\tau].\]
Then the result follows, since, by invoking that $s$ is nonnegative, Lebesgue's dominated convergence theorem gives
\[\lim_{t\to\infty}\int_t^\infty s(y(\tau),u(\tau))\dx[\tau]=0.\]
\end{proof}
Note that, up to this point, we have not dealt with the {optimal control} itself (i.e., a control input that attains the infimum of the cost functional), but only with the optimal {value} of the cost. In general, an optimal control need not exist; one may have only a minimizing sequence of admissible trajectories. This phenomenon already occurs for finite-dimensional systems \cite[Secs.~7\&8]{IlchmannReis2017Outer}. Nevertheless, whenever an optimal control does exist, it can be characterized via the dissipation inequality. We do not state this as a theorem; instead, we record it in the following remark.
 
\begin{remark}[The optimal control]
We have seen that, under the standing assumptions of nonnegativity of the cost functional and state–feedback stabilizability, the value function is a bounded storage function with respect to the supply rate $s$. 
Boundedness implies the existence of a self-adjoint $P\in\Lb(\X)$ such that 
\[
\mathrm{Val}(x_0)=\scprod{ x_0}{Px_0}_\X\quad\text{for all }x_0\in\X.
\]
Moreover, Theorem~\ref{thm:mainbnd} yields a Hilbert space $\Z$ and an operator $K\&L\in\Lb(\dom(A\&B);\Z)$ with dense range such that \eqref{eq:kyp_inf_lurebnd} holds for all $\spvek{x}{u}\in\dom(A\&B)$. 
Applying Proposition~\ref{prop:domScontbnd}, we obtain the following: for $K\&L\in\Lb(\dom(A\&B))$ as in Theorem~\ref{thm:mainbnd}\,\ref{item:dissineq4bnd}, and for all $T>0$, every generalized solution of \eqref{eq:ODEnode} on $[0,T]$ satisfies
\[
  \mathrm{Val}(x(T))-\mathrm{Val}(x(0))
  \;+\; \int_0^{T}s\bigl(y(t),u(t)\bigr)\,\dx[t]
  \;=\; \int_0^{T}\norm*{K\&L\,\spvek{x(t)}{u(t)}}_{\Z}^{2}\,\dx[t].
\]
Invoking Proposition~\ref{prop:xtends0}, which states that $\mathrm{Val}(x(t))\to 0$ as $t\to\infty$ for any generalized trajectory with $u,y\in L^2(\R_{\ge0})$, we conclude that
\[
  \mathrm{Val}(x_0)\;+\;\norm*{K\&L\,\spvek{x(\cdot)}{u(\cdot)}}_{L^{2}(\R_{\ge0};\Z)}^{2}
  \;=\; \mathcal{J}\bigl(y(\cdot),u(\cdot)\bigr).
\]
Hence the quantity $\norm*{K\&L\,\spvek{x(\cdot)}{u(\cdot)}}_{L^{2}(\R_{\ge0};\Z)}^{2}$ measures the suboptimality gap of the control $u(\cdot)$ relative to the optimal cost. In particular, an optimal control exists if and only if the infinite-dimensional differential–algebraic equation
\[
  \spvek{\dot{x}(t)}{0}
  \;=\;
  \begin{bmatrix}A\&B\\[-1mm] K\&L\end{bmatrix}\spvek{x(t)}{u(t)},
  \qquad x(0)=x_0,
\]
admits a solution; any such solution then yields an optimal control. 
In the finite-dimensional case, this closed-loop system coincides with the classical feedback system obtained by the solution of the algebraic Riccati equation \cite{Will71a}.
\end{remark}
Finally, we briefly comment on closely related results in the literature.
\begin{remark}
In \cite{OpmeerStaffans2014,OpmeerStaffans2019}, the minimization of the cost functional
\[
\norm{u(\cdot)}_{\Lp{2}(\R_{\ge0};\U)}^2+\norm{y(\cdot)}_{\Lp{2}(\R_{\ge0};\Y)}^2
\]
subject to a prescribed initial value is considered. In our notation, this corresponds to a supply rate $s(\cdot,\cdot)$ as in \eqref{eq:supply_inf} with $S=0$, $R=\id_\U$, and $Q=\id_\Y$. One essentially obtains the same equations as in \eqref{eq:kyp_inf_lure} for this setting. On the one hand, that cost functional is more restrictive than ours; on the other hand, the class of systems is broader, as \emph{operator nodes} rather than system nodes are treated. Roughly speaking, this framework does not require the state operator $A$ to generate a strongly continuous semigroup; it suffices that $A$ be densely defined with a nonempty resolvent set (e.g., this allows one to treat the backward heat equation).

In \cite{OpmeerStaffans2014,OpmeerStaffans2019}, a frequency-domain viewpoint is adopted, and for a subdomain $\Omega$ of the resolvent set of $A$ intersected with the complex right half-plane, the output (when it belongs to $\Lp{2}(\R_{\ge0};\Y)$) is defined via the inverse Laplace transform of the analytic continuation of
\begin{align*}
\widehat{y}:\quad\Omega&\to\Y,\\
s&\mapsto G(s)\widehat{u}(s)+C\&D\spvek{(s\id_\X -A)^{-1}x_0}{0},
\end{align*}
where $\widehat{u}(\cdot)$ is the Laplace transform of $u(\cdot)$, and $G(\cdot)$ is the transfer function of $\sbvek{A\& B}{C\& D}$, see \cite[Sec.~4.6]{Staf05}. Indeed, by \cite[Lem.~4.7.12]{Staf05}, generalized trajectories satisfy the above frequency-domain relation. We also note that the solution concept in \cite{OpmeerStaffans2014,OpmeerStaffans2019} is considerably more general, not only because operator nodes (rather than system nodes) are treated, but also because the state trajectory need not be continuous as a map $\R_{\ge0}\to\X$. Moreover, no state-feedback stabilizability assumption is imposed; instead, the existence of $\Lp{2}$-trajectories is required for a dense subspace of \(\X\) (which also leads to unbounded value functions).

It is important to observe that those results make essential use of the fact that the $\Lp{2}$-norm of the input enters the cost functional. Without going into details, these results can be applied more or less directly to optimal problems \eqref{eq:optcontr} with full input weighting (that is, $R$ in \eqref{eq:supply_inf} is boundedly invertible). The case of more general cost functionals (even nonnegative ones), however, still requires additional work beyond the scope of this article.
\end{remark}

\section{Examples}\label{sec:ex}
At the outset, note that although the monograph \cite{Staf05} contains almost no PDE examples. Partial differential equations, especially with boundary control, are very naturally described within the system node framework. See, for instance, \cite{FriReRe24,ReisSchaller2024Oseen,PhilippReisSchaller2025,ReisJacobFarkasSchmitz2023}, where boundary-controlled beams, waves, the heat equation, Oseen flows, and Maxwell’s equations are treated.

In selecting the two examples below, we strike a balance between transparency and nontriviality.
The first example is deliberately simple: a transport equation with boundary control and observation. From the author’s perspective, this example is well suited to illustrate the spaces $\V$ and $\W$ from \eqref{eq:Vdef} and \eqref{eq:Wdef}.
Thereafter we consider the heat equation on some Lipschitz domain in $\R^d$ with homogeneous Dirichlet boundary conditions and trivial input (i.e., $\U=\{0\}$). As output we take the Neumann trace and use a supply rate with $R=\id_\U$. We show that, while the system is dissipative with respect to a suitable storage function, no bounded storage function exists in this case.

\subsection{Boundary-controlled transport equation}
Consider the first-order transport equation on $[0,1]$ with rightward flow and boundary control at the inflow point:
\[
\tfrac{\partial}{\partial t}x(t,\xi)=-\tfrac{\partial}{\partial \xi}x(t,\xi),\qquad
x(0,\xi)=x_0(\xi),\quad t\ge0,\ \xi\in[0,1],
\]
together with the input and output
\[
u(t)=x(t,0),\qquad y(t)=\alpha\, x(t,0)+x(t,1),
\]
where $\alpha\in\C$ is a~parameter.
A~formal consideration yields that
\begin{align*}
\ddts \int_0^1 |x(t,\xi)|^2\,\dx[\xi]
&= 2\Re \int_0^1 \overline{x(t,\xi)}\,\tfrac{\partial}{\partial t}x(t,\xi)\,\dx[\xi]\\
&= -2\Re \int_0^1 \overline{x(t,\xi)}\,\tfrac{\partial}{\partial \xi}x(t,\xi)\,\dx[\xi]\\
&= -\Re \int_0^1 \tfrac{\partial}{\partial \xi}\bigl|x(t,\xi)\bigr|^2\,\dx[\xi]\\
&= |x(t,0)|^2 - |x(t,1)|^2= |u(t)|^2 - |y(t)-\alpha\, u(t)|^2. 
\end{align*}
This suggests that $\Sc$ with 
\[\Sc(x(t))=\int_0^1|x(t,\xi)|^2\dx[\xi]=\norm*{x(t,\cdot)}_{\Lp{2}([0,1])}^2\]
is a~storage function with respect to any supply rate functional with
\begin{multline}
  \forall\,u,y\in\C:\quad   s(u,y)=\spvek{y\\[-1mm]}{u}\adjun\sbmat{q}{s}{\overline{s}}{r} \spvek{y\\[-1mm]}{u}\\
  \geq |y-\alpha\,u|^2-|u|^2=\spvek{y\\[-1mm]}{u}\adjun\sbmat{1}{-\alpha}{-\overline{\alpha}}{\;|\alpha|^2-1} \spvek{y\\[-1mm]}{u}\label{eq:transportsupply}
\end{multline}
for $q,r\in\R$, $s\in\C$. Note that \eqref{eq:transportsupply} is equivalent to positive semidefiniteness of the matrix
\begin{equation}
  M:=\left[\begin{smallmatrix}{q-1}&{s+\alpha}\\[1mm]{-\overline{s+\alpha}}\;&{\;\;r+1-|\alpha|^2}\end{smallmatrix}\right].  \label{eq:Mdef}
\end{equation}
In what follows, we show that the system node framework, together with the results obtained in this article, yields exactly the same conclusion as the formal considerations above. To this end, we first formulate the foregoing system in the system node setting.

In a certain sense, the formulation of the operator $A\&B$ for this example is prototypical for system nodes: the action of the PDE (here, the negative spatial derivative) is mirrored by the action of $A\&B$, whereas the domain encodes the boundary condition (namely, that the left boundary value equals the input). Concretely,
\[
\dom(A\&B)=\setdef*{\spvek{x}{u}\in H^1([0,1])\times\C}{x(0)=u},
\qquad
A\&B\spvek{x}{u}=-\tfrac{\partial}{\partial \xi}x.
\]
Moreover, in accordance with the output definition, the output operator is
\[
C\&D\spvek{x}{u}=\alpha\,x(0)+x(1)=\alpha\,u+x(1)
\quad \forall\,\spvek{x}{u}\in\dom(A\&B).
\]
Note that these boundary evaluations are well defined for $x\in H^1([0,1])$ by the Sobolev trace theorem \cite[Thm.~4.12]{AdamFour03}. The same theorem also implies that $C\&D\in\Lb(\dom(A\&B);\C)$. Closedness of $A\&B$ follows from the definition of the weak derivative, whereas, for all $u\in\C$, the constant function $x(\cdot)\equiv u$ fulfills $\spvek{x}{u}\in\dom(A\&B)$. For the property of $\sbvek{A\&B}{C\&D}$ being a~system node on $(\Lp{2}([0,1]),\C,\C)$, it remains to show that main operator generates a~strongly continuous semigroup on $\Lp{2}([0,1])$. This however follows, since $A$ is given by
\[
Ax=-\tfrac{\mathrm d}{\mathrm d\xi}x,\qquad
\dom A=\setdef{x\in H^1([0,1])}{x(0)=0},
\]
and it is shown in that it generated the right-shift semigroup on $\Lp{2}([0,1])$ (see e.g.\ \cite[Ex.~2.8.7]{TucsnakWeiss2009}.\\
In passing, we note that the spaces $\V$ and $\W$ from \eqref{eq:Vdef} and \eqref{eq:Wdef} are given by
\begin{align*}
    \V&=\setdef{x\in H^{1}([0,1])}{\exists \,u\in\C\text{ s.t.\ }x(0)=u}=H^{1}([0,1]),\\[1mm]
    \W&=\setdef{\spvek{x}u\in H^{1}([0,1])\times\C}{x(0)=u\,\wedge\,\tfrac{\partial}{\partial \xi}x\in H^{1}([0,1])}\\
    &=\setdef{\spvek{x}u\in H^{2}([0,1])\times\C}{x(0)=u}.
\end{align*}
We now show that the functional $\Sc:\Lp{2}([0,1])\to\R$, defined by $\Sc(x)=\norm{x}_{\Lp{2}}^{2}$, is a storage function for the system with respect to the supply rate functional in \eqref{eq:transportsupply}. To this end we first observe that the identity operator on $\Lp{2}([0,1])$ is induced by $\Sc$. Further, by using that the definition of the weak derivative implies that
\[\forall\, x\in H^1([0,1]):\quad2\Re\scprod*{x}{\tfrac{\partial}{\partial\xi}x}=|x(1)|^2-|x(0)|^2.\]
Using the latter, we obtain that, for $M$ as in 
\begin{align*}
&\;2\Re\scprod*{Px}{A\&B\,\spvek{x}{u}}_{\X}
  + s\!\left(C\&D\,\spvek{x}{u},\,u\right)\\
&= -2\Re\scprod*{x}{\tfrac{\partial}{\partial \xi}x}_{\Lp{2}}
   + \spvek{\alpha x(0)+x(1)}{x(0)}\adjun
     \begin{bmatrix} q & s \\ \overline{s} & r \end{bmatrix}
     \spvek{\alpha x(0)+x(1)}{x(0)}\\
&= |x(0)|^{2}-|x(1)|^{2}
   + \spvek{\alpha x(0)+x(1)}{x(0)}\adjun
     \begin{bmatrix} q & s \\ \overline{s} & r \end{bmatrix}
     \spvek{\alpha x(0)+x(1)}{x(0)}\\
&= \spvek{\alpha x(0)+x(1)}{x(0)}\adjun
   \left(
     \begin{bmatrix} q & s \\ \overline{s} & r \end{bmatrix}
     - \begin{bmatrix} 1 & -\alpha \\ -\overline{\alpha} & |\alpha|^{2}-1 \end{bmatrix}
   \right)
   \spvek{\alpha x(0)+x(1)}{x(0)}\\
&= \spvek{\alpha x(0)+x(1)}{x(0)}\adjun
   M\,
   \spvek{\alpha x(0)+x(1)}{x(0)}\geq0,
\end{align*}
where the latter holds due to positive semidefiniteness of $M$. Then Theorem~\ref{thm:mainbnd} yields that the expected dissipativity result holds. To determine the dissipation rate according to Theorem~\ref{thm:mainbnd}\,\ref{item:dissineq4bnd}, we first take a~rank revealing factorization of $M$, i.e., $M=L\adjun L$, $L\in\C^{r\times 2}$, $r=\rk M\in\{0,1,2\}$, and we set 
\[K\&L\spvek{x}u=L\spvek{\alpha x(0)+x(1)}{x(0)}=L\sbmat{0}{1}{1}{-\alpha}\spvek{x(0)}{x(1)},\quad {x}\in H^1([0,1]),\;u=x(0).\]
The previous calculations now indeed show that
\eqref{eq:kyp_inf_lurebnd} holds.

\subsection{Heat equation with Neumann observation}

Let $\Omega\subset\R^{d}$, $d\in\N$, be a Lipschitz domain. Consider the heat equation with homogeneous Dirichlet boundary conditions and a Neumann observation:
\[
\begin{aligned}
\tfrac{\partial}{\partial t} x(t,\xi) &= \Delta x(t,\xi), & x(0,\xi) &= x_0(\xi), && t\ge 0,\ \xi\in\Omega,\\
x(t,\xi) &= 0, &
y(t,\xi) &= \frac{\partial x(t,\cdot)}{\partial \nu}\Big|_{\partial\Omega}, && t\ge 0,\ \xi\in\partial\Omega,
\end{aligned}
\]
where $\nu$ denotes the outward unit normal on $\partial\Omega$. Our system is provided with no input (i.e., $\U=\{0\}$), the state space is chosen as $\X=\Lp{2}(\Omega)$, and the main operator is given by the Dirichlet Laplacian, i.e.,
\begin{equation}
\dom(A)=\setdef{x\in H^1_0(\Omega)}{\nabla x\in H_{\operatorname{div}}(\Omega)},\quad Ax=\Delta x,\label{eq:DirLapl}
\end{equation}
where $H_{\operatorname{div}}(\Omega)$ of all elements of $\Lp{2}(\Omega;\C^d)$ whose weak divergence is in $\Lp{2}(\Omega)$. 
The output is given by the Neumann trace of $x$, which, in the sense of \cite[Chap.~20]{Tartar2007Sobolev}, is the weak normal trace of $\nabla x$ and takes values in the fractional Sobolev space
\[
\Y := H^{-1/2}(\partial\Omega) := H^{1/2}(\partial\Omega)\adjun.
\]
Accordingly (and since the input space is trivial, we write $C$ in place of $C\&D$), the output operator is
\[
Cx = \gamma_{\mathrm n}(\nabla x),
\qquad
\gamma_{\mathrm n}\in \Lb\!\big(H_{\operatorname{div}}(\Omega),\,H^{-1/2}(\partial\Omega)\big),
\]
where $\gamma_{\mathrm n}$ denotes the normal trace operator. The system node properties were established in \cite[Sec.~4.1]{PhilippReisSchaller2025} for the more general class of advection–diffusion equations with Dirichlet control; hence they apply here as well.
Since the input space is trivial, the spaces $\V$ and $\W$ from \eqref{eq:Vdef} and \eqref{eq:Wdef} are, for $A$ the Dirichlet Laplacian as in \eqref{eq:DirLapl}, given by
\[
\V=\dom(A),\qquad \W=\dom(A^{2}).
\]
In what follows we consider the supply rate functional
\[
s(y)=-\norm{y}_{H^{-1/2}(\partial\Omega)}^2.
\]
For academic purposes, we seek for a~nonpositive storage function. By Theorem~\ref{thm:main}, the search for such a~storage function reduces to the operator inequality
\[
2\Re\scprod*{\widetilde{P}x}{Ax}_{\V\adjun,\V}-\norm{Cx}_{\Y}^{2}\ge 0
\quad\text{for all }x\in\V,
\]
which is essentially a Lyapunov inequality in the sense of \cite[Thm.~5.1.1]{TucsnakWeiss2009} for $-\widetilde{P}$. 
This implies that, if there were a bounded nonpositive solution $\widetilde{P}=P\in\Lb(\Lp{2}(\Omega))$, the system would be admissible in the sense of \cite[Def.~4.6.4]{TucsnakWeiss2009}. 
However, as shown in \cite{PreusslerSchwenninger2024LpAdmissibility}, the heat equation with Dirichlet boundary control is not admissible in the sense of \cite[Def.~4.6.1]{TucsnakWeiss2009}; and a duality argument then yields a contradiction for our setup. 
Consequently, no bounded and nonpositive storage function exists for this problem. However, the following argumentation shows that an unbounded storage function exists: Let $\fA$ be the semigroup generated by $A$. Then, by exponential stability of the Dirichlet Laplacian,
the mapping
\begin{align*}
\Psi:\quad\dom(A)&\to\Lp{2}(\R_{\ge0};\Y),\\
x&\mapsto C\fA(\cdot)x \end{align*}
is well-defined. Then, by \cite[Prop.~2.10]{ReisSchaller2025LQ}, $\Psi$ is closable as an operator from $\X$ to $\Lp{2}(\R_{\ge0})$. Denote this closure by $\overline{\Psi}:\X\supset\dom(\overline{\Psi})\to \Lp{2}(\R_{\ge0};\Y)$. Then 
\begin{align*}
\Sc:\quad\dom(\Sc)=\dom(\overline{\Psi})&\to\R,\\
x&\mapsto -\norm*{\overline{\Psi}x}^2_{\Lp{2}(\R_{\ge0};\Y)}\end{align*}
is a~closed and nonnegative form with
\[\V=\dom(A)\subset\dom(\overline{\Psi})=\dom(\Sc).\]
The operator
$\widetilde{P}\in\Lb(\dom(\Sc),\dom(\Sc)\adjun)$ such that  \eqref{eq:Pdual} holds for $h$ as in \eqref{eq:hsesq} then fulfills
\[\forall\,x_1,x_2\in\dom(\Sc):\quad \scprod{\widetilde{P}x_1}{x_2}_{\dom(\Sc)\adjun,\dom(\Sc)}=-\scprod{\overline{\Psi}x_1}{\overline{\Psi}x_2}_{\Lp{2}(\R_{\ge0};\Y)}.\]
Further, by $\fA Ax=\ddts \fA x$ for all $x\in\dom(A)$, we obtain that
\[\forall\,x\in\dom(A^2):\,\Psi Ax=\ddts\Psi x.\]
Invoking the latter, we obtain, by integration by parts, that for all $x\in\W=\dom(A^2)$,
\begin{align*}
2\Re\scprod*{\widetilde{P}x}{Ax}_{\V\adjun,\V}-\norm{Cx}_{\Y}^{2}
&= -2\Re\scprod*{\overline{\Psi}x}{\overline{\Psi}Ax}_{\Lp{2}(\R_{\ge0};\Y)} - \norm{Cx}_{\Y}^{2}\\
&= -2\Re\scprod*{{\Psi}x}{{\Psi}Ax}_{\Lp{2}(\R_{\ge0};\Y)} - \norm{Cx}_{\Y}^{2}\\
&= -2\Re\scprod*{{\Psi}x}{\ddts{\Psi}x}_{\Lp{2}(\R_{\ge0};\Y)} - \norm{Cx}_{\Y}^{2}\\
&= \norm*{(\Psi x)(0)}_{\Y}^{2} - \norm{Cx}_{\Y}^{2}.
\end{align*}
However, by $\fA(0)=\id_\X$, we have
\[\forall\,x\in\dom(A^2):\quad \big({\Psi}x\big)(0)=Cx.\]
This gives
\[    2\Re\scprod*{\widetilde{P}x}{Ax}_{\V\adjun,\V}-\norm{Cx}_{\Y}^{2}=0\geq0,\]
which shows that $\Sc$ is a~storage function.

\begin{appendices}
\section{Auxiliary results on system nodes}
We collect here a few auxiliary results on system nodes. Thematically, placement at the end of Section~\ref{sec:sysnodes} would also be natural, since they rely on the spaces $\V$ and $\W$ introduced in \eqref{eq:Vdef} and \eqref{eq:Wdef}, resp. We have relegated them to the appendix to keep the exposition focused: the results are technically self-contained, used only at a few specific points, and would otherwise distract from the central focus on the dissipation inequality.

We start with a~lemma which states that the inclusions $\W\subset\dom(A\&B)\subset\X\times \U$ are dense. To this end, let us note that, for a~system node $\sbvek{A\& B}{C\& D}$ on $(\X,\U,\Y)$, we can define the Hilbert space $\X_{-1}$ by the completion of $\X$ with respect to the norm $\norm{x}_{\X_{-1}}:=\norm{(\alpha \id_\X-A)^{-1}x}_\X$ for some $\alpha$ in the resolvent of $A$. The operator $A$ extends to closed and densely defined operator $A_{-1} : \X_{-1}\supset\dom (A_{-1}) = \X\to \X_{-1}$, with the same spectrum as $A$ \cite[Prop.~2.10.3]{TucsnakWeiss2009}.
Further, there exists an operator $B\in \Lb(\U,\X_{-1})$ such that $[A_{-1}\ B]\in \Lb(\X\times\U,\X_{-1})$ is an extension of $A\& B$. The domain of $A\&B$ (equally: the domain of $S$) satisfies
 \begin{equation}
 \dom(A\&B)=\setdef{\spvek xu \in \X\times \U}{A_{-1}x+Bu\in \X},
 \label{eq:ABdom}\end{equation}
 see \cite[Def.~4.7.2 \& Lem.~4.7.3]{Staf05}. 

\begin{lemma}\label{eq:Vcore}
Let $\sbvek{A\& B}{C\& D}$ be a~system node on $(\X,\U,\Y)$. Then     
$\W$ as in \eqref{eq:Wdef} is a~core of $A\&B$, and $\dom(A\&B)$ is dense in $\X\times \U$.
\end{lemma}
\begin{proof}
Since $A$ generates a strongly continuous semigroup, it is densely defined \cite[Chap.~2, Thm.~1.4]{EngeNage00}. Hence, a~combination of \eqref{eq:ABdom} with property \ref{def:sysnode3} from Definition~\ref{def:sysnode} implies that $A\&B$ is densely defined.\\
It remains to prove that $\W$ is a~core of $A\&B$. Since $A$ generates a~strongly continuous semigroup, it follows from \cite[Chap.~2, Thm.~1.10]{EngeNage00} that  $\lambda \id_\X-A$ is bijective for some $\lambda\in\C$. Consequently, the operator
\begin{align*}
    T:\quad\dom(A\&B)&\to \X\times\U,\\
    \spvek{x}{u}&\mapsto \spvek{A\&B\spvek{x}{u}-\lambda x\\[-1mm]}{u}
\end{align*}
is bijective. It can be further seen that the inverse of $T$ fulfills
\begin{align*}
    T^{-1}:\quad\X\times\U &\to \dom(A\&B),\\
    \spvek{x}{u}&\mapsto \spvek{(A-\lambda \id_\X)^{-1}x-(A_{-1}-\lambda \id_\X )^{-1}Bu}{u}.
\end{align*}
Now assume that $\spvek{x}u\in\dom(A\&B)$. To prove the desired result, we have to show that there exists a~sequence $\big(\spvek{x_n}{u_n}\big)$  in $\V$ which converges in $\dom(A\&B)$ to $\spvek{x}u$. Since $\dom(A\&B)$ is dense in $\X\times\U$, there exists a~sequence $\big(\spvek{z_n}{w_n}\big)$ in $\dom(A\&B)$ which converges in $\X\times\U$ to $T\spvek{x}u$. Define
\[
\big(\spvek{x_n}{u_n}\big):=T^{-1}\big(\spvek{z_n}{w_n}\big).\]
Then it follows immediately that $\big(\spvek{x_n}{u_n}\big)$ converges in $\dom(A\&B)$ to $\spvek{x}{u}$. Further, the structure of $T^{-1}$ yields that
\[\forall\,n\in\N:\quad A\&B\spvek{x_n}{u_n}=z_n+\lambda x_n\in\V.\]
The latter gives $\spvek{x_n}{u_n}\in \W$ for all $n\in\N$,
and the result is proven.
\end{proof}
Next we prove two lemmas on smooth trajectories of systems governed by system nodes.

\begin{lemma}\label{lem:smoothsol}
Let $\sbvek{A\& B}{C\& D}$ be a~system node on $(\X,\U,\Y)$, let $T>0$, and assume that $(x(\cdot),u(\cdot),y(\cdot))$ is a classical trajectory of \eqref{eq:ODEnode} on $[0,T]$ with $x(\cdot)\in\conC^{2}([0,T];\X)$ and $u(\cdot)\in\conC^{1}([0,T];\U)$. Then, for $\V$ as defined in \eqref{eq:Vdef}, and $\W$ as defined in \eqref{eq:Wdef},
\[x(\cdot)\in \conC^{1}([0,T];\V)\;\text{ and }\spvek{x(\cdot)}{u(\cdot)}\in \conC([0,T];\W).\]
\end{lemma}
\begin{proof}
Closedness of $\Sc$ together with smoothness of $x(\cdot)$, $u(\cdot)$ and $y(\cdot)$ yields that a~formal differentiation of \eqref{eq:ODEnode} is possible, whence, $(\dot{x}(\cdot),\dot{u}(\cdot),\dot{y}(\cdot))$ is a classical trajectory of  \eqref{eq:ODEnode} on $[0,T]$. In particular, 
\[\spvek{\dot{x}(\cdot)}{\dot{u}(\cdot)}\in\conC([0,T];\dom(A\&B)),\]
and, consequently, $\dot{x}(\cdot)\in \conC([0,T];\V)$. This implies that $x(\cdot)\in \conC^{1}([0,T];\V)$, and, by \[A\&B\spvek{x(\cdot)}{u(\cdot)}=\dot{x}(\cdot)\in \conC([0,T];\W),\] we also obtain that 
\[\spvek{x(\cdot)}{u(\cdot)}\in \conC([0,T];\W).\]
\end{proof}
Next we show that for every $\spvek{x_0}{u_0}\in \W$ there exists a~smooth classical trajectory with $x(0)=x_0$ and $u(0)=u_0$. The proof relies heavily on Proposition \ref{prop:solex}.
\begin{lemma}\label{lem:solexsmooth}
Let $\sbvek{A\& B}{C\& D}$ be a~system node  on $(\X,\U,\Y)$, $T\in\R_{>0}$, and let $\V$ and $\W$ be as in \eqref{eq:Vdef} and \eqref{eq:Wdef}.
Then, for all $\spvek{x_0}{u_0}\in \W$, there exists 
a~classical trajectory $(x(\cdot),u(\cdot),y(\cdot))$ of \eqref{eq:ODEnode} on $[0,T]$ with $x(0)=x_0$, $u(0)=u_0$, $x\in \conC^{2}([0,t];\X)$ and $\spvek{x(\cdot)}{u(\cdot)}\in \conC([0,T];\W)$.
\end{lemma}
\begin{proof}
Since $\spvek{x_0}{u_0}\in \W$, there exists some $u_0\in \U$, such that 
\[\spvek{A\&B\spvek{x_0}{u_0}}{u_1}\in\dom(A\&B).\]
By simple interpolation, we obtain that there exists some  
\[u(\cdot)\in \Wkp{3,1}([0,T];\U)\;\text{ s.t.\ }u(0)=u_0\text{ and }\dot{u}(0)=u_1.\]
Then Proposition~\ref{prop:solex} implies that there exist classical trajectories $(x(\cdot),u(\cdot),y(\cdot))$, $(x_1(\cdot),\dot{u}(\cdot),y_1(\cdot))$ on $[0,T]$ with $x(0)=x_0$ and $x_1(0)=A\&B\spvek{x_0}{u_0}$. Then {\cite[Thm.~3.8.2]{Staf05}}
gives $x_1(\cdot)=\dot{x}(\cdot)$. Hence, $x(\cdot)\in\conC^2([0,T];\X)$ and, by further invoking that $u(\cdot)\in \Wkp{3,1}([0,T];\U)\subset \conC^2([0,T];\X)$, an application of Lemma~\ref{lem:smoothsol} leads \[x(\cdot)\in \conC^{1}([0,T];\V)\;\text{ and }\;\spvek{x(\cdot)}{u(\cdot)}\in \conC([0,t];\W),\] the desired result.  
\end{proof}

Next we present a~result on the approximation of generalized trajectories by smooth classical trajectories.
Hereby we consider the concept of `mollifier sequence'.
\begin{definition}[Mollifier sequence]\label{def:molli}
    By a `mollifier sequence', we mean a~sequence $(\alpha_n)_{n\in\N}$ in $\conC^\infty(\R)$, such that, for all $n\in\N$,
\begin{enumerate}[label=(\alph{*})]
  \item $\displaystyle \int_\R \alpha_n(t)\,\dx[t] = 1$,
  \item the support of $\alpha_n $ is contained in $(-1/n,0]$,
  \item $\alpha_n(t)\ge 0 \text{ for all } t\in \R$.
\end{enumerate}
\end{definition}
Such a sequence exists, see 
\cite[p.~110]{Alt16}. Next we show that the convolution of a~generalized trajectory gives a~sequence of classical solutions converges.
The latter is, for $x(\cdot)\in\Lp{1}([0,T];\X)$, 
\begin{equation}
\begin{aligned}\label{eq:convol}
    \alpha_n\ast x:\quad [0,T]&\to\X,\\
    t&\mapsto \int_{-1/n}^t\alpha_n(\tau)x(t-\tau)\dx[\tau],
\end{aligned}    
\end{equation}
where we set $x(t)=0$ for $t>T$.
\begin{lemma}\label{lem:mollconv}
    Let $\sbvek{A\& B}{C\& D}$ be a~system node  on $(\X,\U,\Y)$, let $T\in\R_{>0}$, and assume that $(x(\cdot),u(\cdot),y(\cdot))$ is a generalized trajectory of \eqref{eq:ODEnode} on $[0,T]$. Consider the 
convolutions $x_n:=\alpha_n\ast x:[0,T]\to\X$,
$u_n:=\alpha_n\ast u:[0,T]\to\U$, 
$y_n:=\alpha_n\ast y:[0,T]\to\Y$. Then the following holds:
  \begin{enumerate}[label=(\alph{*})]
 \item\label{lem:mollconv1} For all 
$n\in\N$, $(x_n(\cdot),u_n(\cdot),y_n(\cdot))$ is a~classical and infinitely often differentiable trajectory of \eqref{eq:ODEnode} on $[0,T-1/n]$
\item\label{lem:mollconv2} For all $\varepsilon>0$, the restriction of $\big((x_n(\cdot),u_n(\cdot),y_n(\cdot))\big)$ to $[0,T-\varepsilon]$ converges to the restriction of $(x(\cdot),u(\cdot),y(\cdot))$ to  $[0,T-\varepsilon]$ in the topology of $\conC([0,T-\varepsilon];\X)\times \Lp{2}([0,T-\varepsilon];\U)\times \Lp{2}([0,T-\varepsilon];\Y)$.
\item\label{lem:mollconv3} If, moreover, for some Hilbert space $\widetilde{\X}\subset\X$, it holds
that $x(\cdot)\in\conC([0,T];\widetilde{\X})$, then the sequence $(x_n(t))$ converges in $\V$ to $x(t)$ for all $t\in[0,T)$.  
\end{enumerate}
\end{lemma}
In the following, we present some convergence results for convolutions from \cite{Brezis2011}, which are originally stated for the scalar-valued case. A careful inspection of the proofs, however, shows that they also remain valid in the Hilbert space–valued setting.
\begin{proof}\
\begin{enumerate}[label=(\alph{*})]
\item Smoothness of $x_n(\cdot)$, $u_n(\cdot)$ and $y_n(\cdot)$ is a~classical result, see e.g.\ \cite[Sec.~4.14]{Alt16}.
Since we have a~generalized trajectory, there exists a~sequence $(\tilde{x}_m(\cdot),\tilde{u}_m(\cdot),\tilde{y}_m(\cdot))$ which converges to $(x(\cdot),u(\cdot),y(\cdot))$ in the topology of $\conC([0,T];\X)\times \Lp{2}([0,T];\U)\times \Lp{2}([0,T];\Y)$. A~differentiation of the convolution integrals yields that, for all $n,m\in\N$, 
$(\alpha_n\ast\tilde{x}_m(\cdot),\alpha_n\ast\tilde{u}_m(\cdot),\alpha_n\ast\tilde{y}_m(\cdot))$ is a~classical trajectory of 
of \eqref{eq:ODEnode} on $[0,T-1/n]$. Further, by \cite[Prop~4.21 \& Thm.~4.22]{Brezis2011}, we have for all $n\in\N$ that the limit $m\to\infty$
of
\[(\alpha_n\ast\tilde{x}_m(\cdot),\alpha_n\ast\tilde{u}_m(\cdot),\alpha_n\ast\tilde{y}_m(\cdot))\]
in the topology of $\conC([0,T];\X)\times \Lp{2}([0,T];\U)\times \Lp{2}([0,T];\Y)$ is given by $(x_n(\cdot),u_n(\cdot),y_n(\cdot))$.
Hence, for all $n\in\N$, $(x_n(\cdot),u_n(\cdot),y_n(\cdot))$ is a~generalized trajectory 
of \eqref{eq:ODEnode} on $[0,T-1/n]$. Now invoking smoothness of $(x_n(\cdot),u_n(\cdot),y_n(\cdot))$ we even have that  
it is a~classical trajectory 
of \eqref{eq:ODEnode} on $[0,T-1/n]$.
\item This follows by an application of \cite[Prop~4.21 \& Thm.~4.22]{Brezis2011}.
\item This again follows from \cite[Prop~4.21]{Brezis2011}.
\end{enumerate}
 \end{proof}

\end{appendices}

\section*{Statements and Declarations}

\subsection*{Funding}
The author gratefully acknowledges funding from the Deutsche Forschungsgemeinschaft (DFG, German Research Foundation), Project-ID 531152215, CRC 1701 ``Port-Hamiltonian Systems''.

\subsection*{Author Contributions}
All authors contributed equally to this work.

\subsection*{Competing Interests} The author reports no conflicts of interest.


\bibliography{refs}

\end{document}